\documentclass[11pt]{article}
\usepackage[utf8]{inputenc}

\usepackage{amsmath}
\usepackage{amssymb}

\usepackage{amsthm}
\theoremstyle{definition}
\newtheorem{theorem}{Theorem}[section]
\newtheorem{definition}[theorem]{Definition}
\newtheorem{proposition}[theorem]{Proposition}
\newtheorem{corollary}[theorem]{Corollary}
\newtheorem{remark}[theorem]{Remark}
\newtheorem{lemma}[theorem]{Lemma}

\usepackage[ruled, linesnumbered]{algorithm2e}
\SetKwInput{KwSample}{Sample}

\usepackage[T1]{fontenc}
\usepackage{lmodern}

\usepackage{tikz}
\usepackage{pgfplots}
\pgfplotsset{compat=1.18}
\usetikzlibrary{arrows.meta}
\usetikzlibrary{backgrounds}
\usepgfplotslibrary{patchplots}
\usepgfplotslibrary{fillbetween}
\pgfplotsset{%
    layers/standard/.define layer set={%
        background,axis background,axis grid,axis ticks,axis lines,axis tick labels,pre main,main,axis descriptions,axis foreground%
    }{
        grid style={/pgfplots/on layer=axis grid},%
        tick style={/pgfplots/on layer=axis ticks},%
        axis line style={/pgfplots/on layer=axis lines},%
        label style={/pgfplots/on layer=axis descriptions},%
        legend style={/pgfplots/on layer=axis descriptions},%
        title style={/pgfplots/on layer=axis descriptions},%
        colorbar style={/pgfplots/on layer=axis descriptions},%
        ticklabel style={/pgfplots/on layer=axis tick labels},%
        axis background@ style={/pgfplots/on layer=axis background},%
        3d box foreground style={/pgfplots/on layer=axis foreground},%
    },
}
\usepackage{adjustbox}
\usepackage{subcaption}

\newcommand{\setto}{\rightrightarrows}

\newcommand{\lparan}{\symbol{40}}
\newcommand{\rparan}{\symbol{41}}

\newcommand{\bbR}{\mathbb{R}}
\newcommand{\bbB}{\mathbb{B}}
\newcommand{\bbN}{\mathbb{N}}

\newcommand{\bbRn}{\mathbb{R}^n}
\newcommand{\bbRm}{\mathbb{R}^m}
\newcommand{\bbRnxn}{\mathbb{R}^{n {\times} n}}
\newcommand{\bbRnxm}{\mathbb{R}^{n {\times} m}}
\newcommand{\bbRmxn}{\mathbb{R}^{m {\times} n}}

\newcommand{\xbar}{\bar{x}}
\newcommand{\ybar}{\bar{y}}

\newcommand*{\diff}{\mathop{}\!\mathrm{d}}

\DeclareMathOperator{\Id}{I}
\DeclareMathOperator{\img}{img}
\DeclareMathOperator{\dist}{dist}
\DeclareMathOperator{\cnvx}{\overline{conv}}
\DeclareMathOperator{\proj}{proj}
\DeclareMathOperator*{\minim}{minimize}
\DeclareMathOperator*{\argmin}{argmin}
\DeclareMathOperator{\tr}{tr}
\DeclareMathOperator{\st}{subject\ to}
\newcommand{\minimize}[3]{\begin{array}{ll}%
  \displaystyle{\minim_{#1}}\quad& #2 \\%
  \st & #3%
  \end{array}}

\usepackage{hyperref}
\usepackage{url}
\usepackage[backend=biber,giveninits=true,maxbibnames=9,maxcitenames=6]{biblatex}
\DeclareFieldFormat{sentencecase}{\MakeSentenceCase*{#1}}

\renewbibmacro*{title}{%
  \ifthenelse{\iffieldundef{title}\AND\iffieldundef{subtitle}}
    {}
    {\ifthenelse{\ifentrytype{article}\OR\ifentrytype{inbook}%
      \OR\ifentrytype{incollection}\OR\ifentrytype{inproceedings}%
      \OR\ifentrytype{inreference}}
      {\printtext[title]{%
        \printfield[sentencecase]{title}%
        \setunit{\subtitlepunct}%
        \printfield[sentencecase]{subtitle}}}%
      {\printtext[title]{%
        \printfield[titlecase]{title}%
        \setunit{\subtitlepunct}%
        \printfield[titlecase]{subtitle}}}%
     \newunit}%
  \printfield{titleaddon}}
\bibliography{../refs}

\usepackage{tikz}
\usepackage{pgfplots}
\pgfplotsset{compat=1.18}
\usetikzlibrary{arrows.meta}
\usetikzlibrary{backgrounds}
\usepgfplotslibrary{patchplots}
\usepgfplotslibrary{fillbetween}
\pgfplotsset{%
    layers/standard/.define layer set={%
        background,axis background,axis grid,axis ticks,axis lines,axis tick labels,pre main,main,axis descriptions,axis foreground%
    }{
        grid style={/pgfplots/on layer=axis grid},%
        tick style={/pgfplots/on layer=axis ticks},%
        axis line style={/pgfplots/on layer=axis lines},%
        label style={/pgfplots/on layer=axis descriptions},%
        legend style={/pgfplots/on layer=axis descriptions},%
        title style={/pgfplots/on layer=axis descriptions},%
        colorbar style={/pgfplots/on layer=axis descriptions},%
        ticklabel style={/pgfplots/on layer=axis tick labels},%
        axis background@ style={/pgfplots/on layer=axis background},%
        3d box foreground style={/pgfplots/on layer=axis foreground},%
    },
}
\usepackage{adjustbox}

\providecommand{\sep}{, }
\providecommand{\keywords}[1]{\textbf{\textit{Keywords:}} #1}

\author{Titus Pinta}
\title{A Newton-type Method for Non-smooth Under-determined Systems of Equations}

\begin{document}
\maketitle
\begin{abstract}
  We study a variant of Newton's algorithm applied to
  under-determined systems of non-smooth equations. The notion of
  regularity employed in our work is based on Newton differentiability, which
  generalizes semi-smoothness. The classic notion of
  Newton differentiability does not suffice for our purpose, due
  to the existence of multiple zeros and as such we extend it
  to uniform Newton differentiability.
  In this context, we can show that the distance between the iterates
  and the set of zeros of the system decreases super-linearly. For the
  special case of smooth equations, the assumptions of our algorithm are
  simplified.
  Finally, we provide
  some numerical examples to showcase the behavior of our proposed
  method. The key example is a toy model of complementarity constraint
  problems, showing that our method has great application potential
  across engineering fields.
\end{abstract}

\keywords{Newton's Method\sep%
  Under-determined systems\sep%
  Higher Order Methods\sep%
  Nonsmooth Equation
}

\section{Introduction}
In our work, we are interested in solving a general (non-smooth) equation,
\begin{equation*}
  G(x) = 0
\end{equation*}
with $G:U \subseteq \bbRm \to \bbRn$ and $m > n$. By solving such a system, we mean finding a
point $\xbar$ with $G(\xbar)=0$ and not approximating the entire manifold of solutions.
A comprehensive exposition of iterative methods
for solving such problems can be seen in~\cite{OrtRhe00Iter}.

Such non-smooth systems arise often as non-smooth reformulations of
under-determined problems with complementarity constraints. In general, imposing
state space constraints to a physical system yields, after a suitable
discretization, to complementarity constraints and as such our method can
find useful applications in the design of over-parameterized mechanical
devices.

The principal structure of the classic Newton's method can be
adapted to under-determined problems by iteratively solving linear approximations
of the original problem. Such a linear equation yields and affine subspace as
its solution, and the problem that arises consists in picking a
point from this affine subspace. These types of methods have been analyzed
in~\cite{PolTre17Solv,VatBor24Conv}.

This work focuses on a natural choice of such a point, namely the projection
of the current iterate onto the affine subspace. In contrast,~\cite{PolTre17Solv}
works with the projection of 0 onto the affine subspace, i.e.~the element
with minimal norm. When the Jacobian of $G$ has full rank, our method
boils down to the standard Newton's method, with the inverse of the Jacobian
replaced by the Moore-Penrose pseudo-inverse. A similar algorithm has been
successfully applied by Ben-Israel~\cite{Ben66Anew} to solve overdetermined
problems.

The biggest departure of our paper from these previous works stems from
our method's ability to handle non-smooth equations.
In order to handle non-smooth problems, we use the notion of regularity
introduced by Qi in~\cite{Qi_96Cdif}. This notion, originally called
C-differentiability, stems from semi-smoothness~\cite{Qi_Sun93Anon},
the original idea of extending Newton's method to non-smooth problems. In
our work, we have chosen to present this regularity notion under the name
of Newton differentiability, a terminology used also in~\cite{BroUlb22Newt}.

Our contribution is the successful application of the non-smooth Newton-type
techniques, using Newton differentiability, to under-determined systems of
equations, extending the applicability of previously known algorithms.

\subsection{Notation, Definitions and Basic Properties}
We denote the open, and closed, ball at $x$ with radius $r$ by
$\bbB_r(x)$, and $\bbB_r[x]$ respectively. Next, we consider $U$ and $V \subseteq U$ two
non-empty subsets of $\bbRm$.

For reference we define some concepts from linear algebra.
\begin{definition}
  Given a matrix $A \in \bbRnxm$, the kernel of $A$ is
  $\bbRn \supseteq \ker A = \{x \in \bbRn~|~ A x = 0\}$, the image of $A$ is
  $\bbRm \supseteq \img A = \{y \in \bbRm~|~ \exists x \in \bbRn, Ax = y\}$. For a closed and convex
  subset $U \subseteq \bbRn$, the projection is defined as $\proj_U:\bbRn \to U \subseteq \bbRn$,
  $\proj_U x = \argmin_{y \in U}\|x - y\|^2$. A vector $x \in \bbRn$ is orthogonal
  to a vector subspace $U \subseteq \bbRn$, denoted $x \perp U$, if $\forall u \in U$ $\langle x, u \rangle = 0$,
  and two vector subspaces $U$ and $V$ are orthogonal, denoted $U \perp V$ if
  $\forall u \in U$, $u \perp V$.
\end{definition}

\begin{definition}
  Given a differentiable function $G:U \subseteq \bbRm \to \bbRn$
  the matrix $\nabla G(x) \in \bbRnxm$, ${\nabla G(x)}_{ij} = \tfrac{\partial G_i}{\partial x_j},$
  $i \in \{1, \dots, n\}, j \in \{1, \dots, m\}$ is called the {\em Jacobian of $G$ at $x$}.
\end{definition}

\begin{definition}
  A function $G:U \subseteq \bbRm \to \bbRn$ has a {\em directional derivative\/} at $\xbar \in U$
  where $U$ is a neighborhood of $\xbar$ if for all $y \in U$, and $d = y - x$ the
  limit
  \begin{equation*}
    \lim_{t \downarrow 0}\frac{G(x + td) - G(x)}{t} = G'(x; d)
  \end{equation*}
  exists and is finite. A function that has a directional derivative at all
  points $x \in U$ is called {\em Gâteaux differentiable}.
\end{definition}

In the remainder of this section we recall for completeness some of the algebraic
properties of the Moore-Penrose pseudo-inverse. For a thorough exposition of
these ideas see~\cite{Str05Line}.
\begin{definition}[Pseudo-Inverse]
  Let $A \in \bbRnxm$ with full rank. The matrix defined as
  \begin{equation*}
    A^{+} = {(A^T A)}^{-1}A^T,
  \end{equation*}
  if $m \le n$ and as
  \begin{equation*}
    A^{+} = A^T {({A} A^T)}^{-1},
  \end{equation*}
  if $m > n$, is called the {\em pseudo-inverse\/} of $A$.
\end{definition}
\begin{proposition}\label{prop:psudo inverse properties}
  If $A \in \bbRmxn$, then
  \begin{enumerate}
  \item $A {A}^{+} A = A$,
  \item $A^{+} A {A}^{+} = A^{+}$,
  \item $\proj_{\ker A} = \Id - A^{+}A$,
  \item $\proj_{\ker A^T} = \Id - A {A}^{+}$,
  \item $\ker A \perp \img A^T$.
  \end{enumerate}
\end{proposition}

\section{Newton Differentiability}
The key notion of regularity for non-smooth problems for our work is that of
Newton differentiability, introduced by Qi in~\cite{Qi_96Cdif}. In this section
we will introduce this important notion and we will extend it to a uniform
version required by the non uniqueness of the zeros of under-determined
systems. We then prove that this notion not only covers the smooth case, but
it is also a natural generalization of the semi-smooth case.
\begin{definition}[Newton Differentiability]%
\label{d:Newton Diff Euclidean}
  A function $G:U \subseteq \bbRm \to \bbRn$ is called
  {\em pointwise weakly Newton differentiable at $\xbar \in U$\/} if there exists a
  set-valued mapping $\mathcal{H}G:U \setto \bbRnxm$ such that
  \begin{equation*}
    \lim_{x \to \xbar}\sup_{H \in \mathcal{H}G(x)}\frac{\|G(x) - G(\xbar) - H(x - \xbar)\|}{\|x - \xbar\|} < \infty.
  \end{equation*}
  Furthermore, if
  \begin{equation*}
    \lim_{x \to \xbar}\sup_{H \in \mathcal{H}G(x)}\frac{\|G(x) - G(\xbar) - H(x - \xbar)\|}{\|x - \xbar\|} = 0,
  \end{equation*}
  the function is called {\em pointwise Newton differentiable at $\xbar$}.
\end{definition}
Analogously, we define uniform Newton differentiability.
\begin{definition}[Uniform Newton Differentiability]%
\label{d:uniform Newton Diff}\label{def:unif-newton-diff}
  A function $G:U \subseteq \bbRm \to \bbRn$ is called
  {\em uniformly weakly Newton differentiable on $V \subseteq U$\/} if there exists a set
  valued mapping $\mathcal{H}G:U \setto \bbRnxm$ and a constant $M$ such that for every $\varepsilon > 0$ there
  exists a $\delta$ such that for all $x \in U$ and all $y \in V$ with $\|x - y\| \le \delta$,
  \begin{equation*}
    \sup_{H \in \mathcal{H}G(x)}\frac{\|G(x) - G(y) - H(x - y)\|}{\|x - y\|} < M + \varepsilon.
  \end{equation*}
  Furthermore, if
  \begin{equation*}
    \sup_{H \in \mathcal{H}G(x)}\frac{\|G(x) - G(y) - H(x - y)\|}{\|x - y\|} < \varepsilon,
  \end{equation*}
  the function is called {\em uniformly Newton differentiable on $V \subseteq U$}.
\end{definition}

\begin{remark}
  The mapping $\mathcal{H}G$ is called a {\em Newton differential of $G$}. Such a mapping
  needs not to be unique.
\end{remark}

\begin{remark}
  Definition~\ref{d:Newton Diff Euclidean} resembles that of Fréchet
  differentiability, with the
  difference stemming from the fact that in Newton differentiability, the
  differential in the limit is evaluated at $x$ and not at $\xbar$.
\end{remark}

\subsection{Examples}
In this subsection we show that the classical regularity tools
employed in the analysis of Newton-type methods can be subsumed by
Newton differentiability.
First, the Fréchet differential of a smooth function can be seen as a Newton
differential.
\begin{proposition}[Newton Differentiability of Smooth Functions]%
  \label{thm:smoothness-implies-newton}
  On an open and convex subset $U \subseteq \bbRm$,
  let $G:U \to \bbRn$, $G \in \mathcal{C}^1$ and $x_0 \in U$
  Then for any $\rho > 0$ with $\bbB_\rho[x_0] \subseteq U$, $G$ is uniformly Newton
  differentiable on $\bbB_\rho[x_0]$ with a Newton differential $\mathcal{H}(x) := \{\nabla G(x)\}$.
\end{proposition}

\begin{proof}
  Let $\rho > 0$ be arbitrary chosen such that $\bbB_\rho[x_0] \subseteq U$. Because $U$ is
  an open set, we know that there is $\delta_0 > 0$ such that $\bbB_{\rho + \delta_0}[x_0] \subseteq U$.
  Because $\nabla G$ is continuous, we can deduce, using the Haine-Cantor theorem
  that is uniformly continuous on $\bbB_{\rho + \delta_0}[x_0]$. Next choose $\varepsilon > 0$ and
  because of the uniform continuity, we know that there exists $\delta > 0$ such that
  for any $x, y \in \bbB_{\rho + \delta_0}[x_0]$ with $\|x - y\| \le \delta$,
  \begin{equation*}
    \|\nabla G(x) - \nabla G(y)\| \le \varepsilon.
  \end{equation*}

  It is clear that if $y \in \bbB_\rho[x_0]$ and $x \in U$ with $\|x - y\| \le \min\{\delta, \delta_0\}$,
  then $x \in \bbB_{\rho + \delta_0}[x_0]$ and for any
  $t \in [0, 1]$ $x, t(x- y) \in \bbB_{\rho + \delta}[x_0]$ and $\|x - t(x - y)\| \le \delta$, so
  \begin{equation}\label{eq:thm-smoothness-implies-newton-bound-step-1}
    \|\nabla G(x + t(x - y)) - \nabla G(x)\| \le \varepsilon.
  \end{equation}

  Because $U$ is convex, we can use the fundamental theorem
  of calculus to compute for any $y \in \bbB_\rho[x_0]$  and for any $x \in U$ with
  $\|x - y\| \le \min\{\delta, \delta_0\}$,
  \begin{align*}
    \|G(x) &- G(y) - \nabla G(x)(x - y)\| \nonumber \\
    &= \left\|\int_{0}^{1}\nabla G(x + t(x - y))(x - y)\diff t - \nabla G(x)(x - y)\right\|\nonumber \\
    &\le \int_{0}^{1}\|\nabla G(x + t(x - y)) - \nabla G(x)\|\|x - y\| \le \frac{\varepsilon}{2}\|x - y\|,
  \end{align*}
  where the last inequality follows
  from~\eqref{eq:thm-smoothness-implies-newton-bound-step-1}.

  The fact that $\rho$ was arbitrary chosen completes the proof.
\end{proof}

Semi-smooth functions have been introduced by Miffin in~\cite{Mif77Semi} and
have been used for Newton-type methods by Qi and Sun~\cite{Qi_Sun93Anon}.
A comprehensive analysis of semi-smooth maps can be found in~\cite{Mov14Nons}.
This idea laid the basis for the development of Newton differentiability,
and as such it should come as no surprise that semi-smoothness fits neatly
into the theory developed in this paper.

We recall the definition of the Clarke generalized Jacobian of a map $G$.
\begin{definition}[Clarke Jacobian]
  Let $G:U \subseteq \bbRm \to \bbRn$ be a Lipschitz continuous function, and $D \subseteq U$ the full
  measure subset (as per Rademacher's theorem)
  where $G$ is differentiable. The set-valued map $\partial^C G:U \setto \bbRnxm$ defined by
  \begin{equation*}
    \partial^C G(x) =
    \cnvx \left \{ H \in \bbRnxm~|~ \exists {\{x^k\}}_{k \in \bbN} \in D, \lim_{k \to \bbN} \nabla G(x^k) = H \right \}
  \end{equation*}
  is called the {\em Clarke Jacobian}.
\end{definition}

With this notion we can recall the definition of a semi-smooth map.
\begin{definition}[Semi-smoothness]
  A Lipschiz continuous function $G:U \subseteq \bbRm \to \bbRn$ is called
  {\em pointwise semi-smooth at $\xbar$\/} if
  for any sequences ${\{t_k\}}_{k \in \bbN} > 0$, ${\{d^k\}}_{k \in \bbN}$ convergent, and
  ${\{H^k\}}_{k \in \bbN} \in \partial^C G(\xbar + t_k d^k)$ the limit $\lim_{k \to \infty} H^k(\lim_{n \to \infty}d_k)$
  exists.
\end{definition}
The next statement, taken from~\cite[Lemma 2.2]{HenOut01Asub} is a
fundamental property of semi-smooth maps.
\begin{proposition}\label{prop:semismooth}
  Let $G:U \subseteq \bbRm \to \bbRn$ be semi-smooth at $\xbar \in U$. Then it is Gâteaux differentiable
  at $\xbar$ and for any $y \in U$ and for any sequence ${\{x^k\}}_{k \in \bbN}$ with
  $\lim_{k \to \infty}x^k = \xbar$, and ${\{H^k\}}_{k \in \bbN} \in \partial^c G(x^k)$,
  \begin{equation*}
    \lim_{k \to \infty}H^k(y - \xbar) = G'(\xbar, y - \xbar).
  \end{equation*}
\end{proposition}

In order to obtain super-linear convergence for Newton-type methods applied to
standard equation systems, a slightly
stronger condition has to be imposed.
\begin{definition}[Semi-smoothness*]
  A Lipschitz continuous function $G:U \subseteq \bbRm \to \bbRn$ is called
  {\em semi-smooth* at $\xbar \in U$\/} if it is semi-smooth at all $x \in U$
  and
  \begin{equation*}
    \lim_{x \to \xbar}\frac{\|G'(x; x - \xbar) - G'(\xbar; x - \xbar)\|}{\|x - \xbar\|} \le \varepsilon.
  \end{equation*}
\end{definition}

For our work we need to expand this traditional definition to its uniform
analogue.
\begin{definition}[Uniform Semi-smoothness*]
  A Lipschitz continuous function $G:U \subseteq \bbRm \to \bbRn$ is called
  {\em uniformly semi-smooth* on $V \subseteq U$\/} if it is semi-smooth at all $x \in U$
  and
  for every $\varepsilon > 0$ there
  exists a $\delta > 0$ such that for all $x \in U$ and all $y \in V$ with $\|x - y\| \le \delta$,
  \begin{equation*}
    \frac{\|G'(x; x - y) - G'(y; x - y)\|}{\|x - y\|} \le \varepsilon.
  \end{equation*}
\end{definition}

The proof of the next proposition is inspired by Shapiro's work~\cite{Sha90Onco}.
\begin{proposition}[Newton Differentiability of Semi-smooth* Functions]%
  \label{thm:semi-smoothness-implies-newton}
  On an open and convex subset $U \subseteq \bbRm$,
  let $G:U \to \bbRn$ be uniformly semi-smooth* on $V \subseteq U$. Then $G$ is
  uniformly Newton differentiable on $V$ with a Newton differential
  $\mathcal{H}(x) := \{\nabla G(x)\}$.
\end{proposition}

\begin{proof}
  Let $\varepsilon > 0$ be arbitrary chosen.
  Consider a vector $v \in \bbRn$ with $\|v\| = 1$ such that
  \begin{equation*}
    \langle v, G(x) - G(y) - G'(x; y - x) \rangle = \|G(x) - G(y) - G'(x; y - x)\|
  \end{equation*}
  and the function $\varphi:[0, 1] \to \bbR$ defined by
  \begin{equation*}
    \varphi(t) = \langle v, G(x) - G(y + t(x - y)) - t G'(x; x - y) \rangle.
  \end{equation*}
  Clearly, $\phi$ is differentiable because $G$ is Gâteaux differentiable and for
  any fixed $t \in [0, 1]$
  \begin{equation*}
    \varphi'(t) = \langle v, G'(y + t(x - y); x - y) - G'(x; x - y) \rangle \le
    \|G'(y + t(x - y); x - y) - G'(x; x - y)\|
  \end{equation*}
  Using the homogenity of the directional derivative, we get
  \begin{equation*}
    \varphi'(t) \le \frac{1}{1-t}\|G'(y + t(x - y); (t - 1)(x - y))
    - G'(x; (t - 1)(x - y))\|,
  \end{equation*}
  and finally, using the uniform semi-smoothness* there exists $\delta > 0$ such
  that if $\|x - y\| \le \delta$ then
  \begin{equation*}
    \frac{\varphi'(t)}{\|x - y\|} \le \frac{\|G'(y + t(x - y); (t - 1)(x - y))
    - G'(x; (t - 1)(x - y))\|}{\|(t - 1)(x - y)\|} \le \varepsilon.
  \end{equation*}
  Taking the supremum over all $t$ yields
  \begin{equation}\label{eq:semismooth-bound on t}
    \sup_{t \in [0, 1]}\frac{\varphi'(t)}{\|x - y\|} \le \varepsilon.
  \end{equation}

  On the other hand, because $\varphi$ is differentiable on $(0, 1)$ we can conclude,
  using the classic mean value theorem that
  \begin{equation*}
    |\varphi(1) - \varphi(0)| \le \sup_{t \in [0, 1]}\varphi'(t).
  \end{equation*}
  Clearly $\varphi(0) = 0$ and $\varphi(1) = \|G(x) - G(y) - G'(x; y - x)\|$, so substituting
  in~\eqref{eq:semismooth-bound on t} yields
  \begin{equation*}
    \frac{\|G(x) - G(y) - G'(x; y - x)\|}{\|x - y\|} \le \varepsilon,
  \end{equation*}
  for all $x$ and $y$ with $\|x - y\| \le \delta$.

  By the definition of the Clarke Jacobian
  \begin{equation*}
    \sup_{H \in \partial^C G(x)} \frac{\|G(x) - G(y) - H(y - x)\|}{\|x - y\|}
    = \lim_{x \to \xbar, \exists \nabla G(x)}\frac{\|G'(\xbar; y - \xbar) - \nabla G(x)(y - \xbar)\|}{\|y - \xbar\|}.
  \end{equation*}
  For the final step of the proof, we use~\ref{prop:semismooth} to conclude that
  \begin{equation*}
    \lim_{x \to \xbar, \exists \nabla G(x)}\frac{\|G'(\xbar; y - \xbar) - \nabla G(x)(y - \xbar)\|}{\|y - \xbar\|} =
    \frac{\|G(x) - G(y) - G'(x; y - x)\|}{\|x - y\|} \le \varepsilon,
  \end{equation*}
\end{proof}

\section{Under-determined Problems}
We now focus on the main object of this work, namely solving the system
of equations $G(x) = 0$ for $G:\bbRm \to \bbRn$ with $m > n$.
This under-determined nature of the
problem implies the existence of a manifold of solutions as opposed to the
singletons studied for the classic Newton's method.

The idea behind our work is to adapt the structure of the classic
Newton's method. We linearize the equation around a given start point,
and then we solve this linear
approximation. The solution of a linear equation is an affine subset, so we have
to pick one point from this subset and then we can repeat this procedure.

To formalize this intuition, we recall that by the definition of \lparan{}weak\rparan{} Newton
differentiability, the Newton differential provides a suitable local linear
approximation to the function. Let $G:\bbRm \to \bbRn$ be uniformly \lparan{}weakly\rparan{} Newton
differentiable on $\mathcal{Z} = \{\xbar~|~G(\xbar) = 0\}$ and $m > n$. Denote
the Newton differential of $G$ by $\mathcal{H}G$. Starting from a point $x^0 \in \bbRm$,
we construct the linear approximation of $G$ using $H \in \mathcal{H}G(x^0)$, yielding the
linear system
\begin{equation}\label{eq:constrian underdetermined linearization}
  G(x^0) + H (x - x^0) = 0.
\end{equation}
Denote the affine subspace of solutions
to~\eqref{eq:constrian underdetermined linearization} by $\mathcal{A}(H^{0}, x^0)$. The
main difficulty of this problem consists in choosing $x^1 \in \mathcal{A}(H^{0}, x^0)$. The
natural choice consists in projecting $x^0$ onto $\mathcal{A}(H^{0}, x^0)$.

\begin{remark}
  In this section we require, $G$ to be defined on the entire $\bbRm$ due to
  the geometric need that $\mathcal{A}$ produces affine subsets.
\end{remark}

\begin{definition}[Under-determined Newton-type Method]%
  \label{def:underdetemined newton method}
  Let $G:\bbRm \to \bbRn$ be uniformly weakly Newton differentiable on $\mathcal{Z} = \{\xbar ~|~ G(\xbar) = 0\}$
  and with $m > n$. The fixed point iteration of the proper (nowhere empty)
  set-valued operator $\mathcal{N}_{\mathcal{H}G}:\bbRm \setto \bbRm$, defined by
  \begin{equation*}
    \mathcal{N}_{\mathcal{H}G} x = \{\proj_{\mathcal{A}(H, x)} x~|~H \in \mathcal{H}G(x)\},
  \end{equation*}
  where
  \begin{equation*}
    \mathcal{A}(x, H) = \{y \in \bbRm~|~ G(x) + H (y - x) = 0\}
  \end{equation*}
  and
  \begin{equation}\label{eq:newton underdetermined iteration}
    x^{k+1} \in \mathcal{N}_{\mathcal{H}G}{x}^k
  \end{equation}
  is called an {\em under-determined Newton-type method}.
\end{definition}
\begin{remark}
  In the degenerate case that $H \in \mathcal{H}G(x)$ does not have full rank, i.e.\ is not
  surjective, $\mathcal{A}(H,x)$ might be empty, thus the algorithm can yield the empty
  set. This case is considered pathological and is indicative of an ill posed
  problem.
\end{remark}

\begin{remark}
  When $n = m$ and $H \in \mathcal{H}G(x)$ has full rank, the set $\mathcal{A}(x, H)$ is a singleton
  and the under-determined Newton-type method coincides with the standard
  Newton-type method.
\end{remark}

\begin{remark}
  The set $\mathcal{A}(x, H)$ is an affine set, so it is closed and convex, hence the
  projector operator onto $\mathcal{A}(x, H)$ is single-valued and non-expansive.
\end{remark}

In Figure~\ref{fig:underdetmined newton explained} one step of the algorithm, using
the Jacobian as the Newton differential, is illustrated. It is interesting to
observe that in this special case the affine approximation to the constraints
produces a set that is parallel to the tangent. Indeed, if $G(x) = 0$, the
set $\mathcal{A}(x, {\nabla G(x)})$ is equal to $\{y \in \bbRm~|~{\nabla G(x)}(y - x) = 0\}$, so the
affine space is orthogonal to $\nabla G(x)$ and tangent to the level set.
Alternatively, we can look at the graph manifold $\{(x, G(x))~|~x \in \bbRm\} \subseteq \bbR^{n + m}$,
together with its tangent plane at $(x, G(x))$, and consider the intersection
of this object with the plane $(x, 0)$.
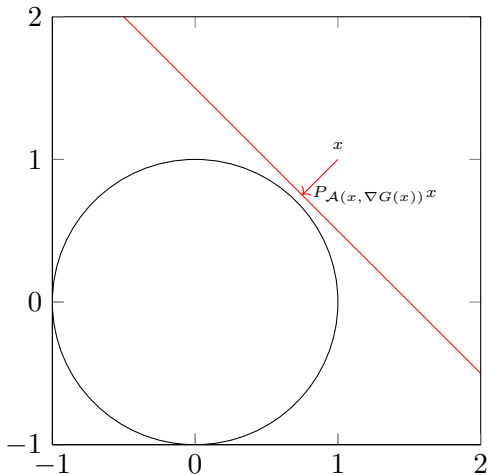
\begin{figure}[ht!]
  \begin{center}
    \begin{tikzpicture}
      \begin{axis}[
        xmin = -1, xmax = 2,
        ymin = -1, ymax = 2,
        zmin = 0, zmax = 1,
        axis equal image,
        view = {0}{90},
        ]
        \draw (axis cs: 0, 0) circle[radius=1];
        \draw (1,1) node[anchor=south] {\tiny{$x$}};
        \draw[red] (-0.5,2)  -- (2,-0.5) node[right,black] {\tiny{$\mathcal{A}(x, \nabla G(x))$}};
        \draw[red,->] (1,1)  -- (0.75,0.75) node[right,black] {\tiny{$P_{\mathcal{A}(x, \nabla G(x))}x$}};
      \end{axis}
    \end{tikzpicture}
  \end{center}
  \caption{One Step of a Newton-type Method for Under-determined Problems}%
  \label{fig:underdetmined newton explained}
\end{figure}

Due to the interaction between the lax nature of the Newton differential and
the geometric behavior of the projections, we will require a further technical
assumption.
\begin{definition}
  Let $G: \bbRm \to \bbRn$ be uniformly weakly Newton differentiable
  on $\mathcal{Z} = \{\xbar \in \bbRm~|~G(x) = 0\}$ with the Newton
  differential of $G$ being $\mathcal{H}G$. This Newton differential is called {\em
    geometrically compatible with $\mathcal{Z}$\/} if there exists $P \ge 1$ such that there
  exists $\delta > 0$ such that for
  all $x \in \bbRm$ with $\dist(x, \mathcal{Z}) < \delta$ and for all $H \in \mathcal{H}G(x)$,
  $\proj_{\mathcal{Z}} \proj_{\mathcal{A}(x, H)}x$ exists, is single-valued and
  \begin{align}\label{eq:linearly geometrically compatible}
    \|\proj_{\mathcal{A}(x, H)}x
    &- \proj_{\mathcal{Z}}\proj_{\mathcal{A}(x, H)}x\| \nonumber \\
    &\le P\|\proj_{\mathcal{Z}}\proj_{\mathcal{A}(x, H)}x - \proj_{\mathcal{A}(x, H)}\proj_{\mathcal{Z}}\proj_{\mathcal{A}(x, H)}x\|.
  \end{align}
\end{definition}
\begin{remark}
  The geometric interpretation of this fact is that the affine approximations
  of $\mathcal{Z}$ produced by $\mathcal{A}$ do not intersect $\mathcal{Z}$ orthogonally, and further this
  intersection angle is uniformly bounded away from $\dfrac{\pi}{2}$.
\end{remark}

Unfortunately, due to the fact that the feasible set is not generally a
singleton, the sequence produced by iterating this algorithm does not converge,
but rather it is a {\em sequence that approaches a set}. The next
theorem states this result in a rigors way. But first, we need a basic
lemma about projections on affine subsets and pseudo inverses.
\begin{lemma}\label{eq:unconstrained basic lemma}
  Let $A \in \bbR^{m \times n}$ be with full rank and $n < m$ and $b \in \bbRn$. Denote
  $\mathcal{S} = \{x \in \bbRm~|~Ax = b\}$. Then for all $x \in \bbRm$
  \begin{equation*}
    AA^{+}(x - \proj_{\mathcal{S}}x) = x - \proj_{\mathcal{S}}x.
  \end{equation*}
\end{lemma}
\begin{proof}
  Because $x - \proj_{S}x$ is orthogonal to $\ker A$, we know that there exists
  $y \in \bbRm$ such
  that $y = x - \proj_{S}x$. From Proposition~\ref{prop:psudo inverse properties},
  we know that $AA^{+}A^T = A^T$, and thus $AA^{+}A^T y = A^T y$.
\end{proof}

\begin{theorem}[Super-linear Convergence of Under-determined Newton Algorithm]%
  \label{thm:superlinear convergence of underdetemined newton}
  Let $G: \bbRm \to \bbRn$ be Lipschitz continuous with constant $L$ and uniform Newton
  differentiable on the feasibility set, $\mathcal{Z} = \{\xbar \in U~|~G(\xbar) = 0\}$.
  Assume that the projection operator onto $\mathcal{Z}$ is Lipschitz continuous with
  constant $L_{\mathcal{Z}}$. Denote
  the Newton differential of $G$ by $\mathcal{H}G$. Assume that $\mathcal{H}G$ is {\em
    geometrically compatible with $\mathcal{Z}$\/} with constant $P$ and that for all $x \in U$, all
  $H \in \mathcal{H}G(x)$ have full rank. Furthermore, assume that the set
  $\bigcup_{x \in U}\{\|H^{+}\|~|~H \in \mathcal{H}G(x)\}$ is bounded by $\Omega \in (0, \infty)$. Then any sequence
  ${\{x^k\}}_{k \in \mathbb{N}}$ with $x^0$ near $\mathcal{Z}$ and generated
  by~\eqref{eq:newton underdetermined iteration}, satisfies
  \begin{equation*}
    \forall k \in \bbN,\quad \dist(x^{k+1}, \mathcal{Z}) \le c^k\dist(x^k, \mathcal{Z}),
  \end{equation*}
  where ${\{c^k\}}_{k \in \bbN}$ is a positive real sequence convergent to $0$.
\end{theorem}
\begin{proof}
  First, we need to prove that $\mathcal{N}_{\mathcal{H}G}$ is a self mapping on a small enough
  neighborhood, $V$, of $\mathcal{Z}$. From the definition of Newton differentiability,
  we can conclude that there exists $c$ with $cP\Omega(1 + L\Omega) < 1$ such that for any
  $\xbar \in \mathcal{Z}$ and for all $x$ in $V$,
  \begin{equation}\label{eq: underdetermined newton diff}
    \sup_{H \in \mathcal{H}G(x)}\|G(x) - H(x - \xbar)\| \le c\|x - \xbar\|.
  \end{equation}
  We can assume that $x$ is close enough to $\mathcal{Z}$ such
  that~\eqref{eq:linearly geometrically compatible} holds.
  Let $H \in \mathcal{H}G(x)$ and $y = \proj_{\mathcal{A}(x, H)}x \in \mathcal{N}_{\mathcal{H}G}(x)$. Using the definition
  of the projector together with~\eqref{eq:linearly geometrically compatible},
  we conclude that $\proj_{\mathcal{Z}} y$ is single valued and that
  \begin{align}\label{eq: underdetemined Newton}
    \dist(y, \mathcal{Z}) &= \|y - \proj_{\mathcal{Z}}y\| \le P\|\proj_{\mathcal{Z}}y - \proj_{\mathcal{A}(x, H)}\proj_{\mathcal{Z}}y\|.
  \end{align}

  Using~\eqref{eq:unconstrained basic lemma} shows
  \begin{align}\label{eq underconsgtin klasdjflkn kjl}
    \|\proj_{\mathcal{Z}}y
    &- \proj_{\mathcal{A}(x, H)}\proj_{\mathcal{Z}}y\|
      = \|H^{+}H(\proj_{\mathcal{Z}}y - \proj_{\mathcal{A}(x, H)}\proj_{\mathcal{Z}}y)\| \nonumber \\
    &\le \|H^{+}\|\|H(\proj_{\mathcal{Z}}y - \proj_{\mathcal{A}(x, H)}\proj_{\mathcal{Z}}y)\|.
  \end{align}
  For the next step we use the definition of $\mathcal{A}(x, H)$ and the fact that
  $\proj_{\mathcal{A}(x, H)} \proj_{\mathcal{Z}}y \in \mathcal{A}(x, H)$, yielding
  \begin{equation*}
    G(x) + H(\proj_{\mathcal{A}(x, H)} \proj_{\mathcal{Z}}y - x) = 0.
  \end{equation*}
  Substituting in~\eqref{eq underconsgtin klasdjflkn kjl} gives
  \begin{align}\label{eq underconsgtin khl lisdlj kjl}
    \|&\proj_{\mathcal{Z}}y - \proj_{\mathcal{A}(x, H)}\proj_{\mathcal{Z}}y\| \nonumber \\
     &\le \|H^{+}\|\|H(\proj_{\mathcal{Z}}y - \proj_{\mathcal{A}(x, H)}\proj_{\mathcal{Z}}y)
       +  G(x) + H(\proj_{\mathcal{A}(x, H)} \proj_{\mathcal{Z}}y - x) \|.
       \nonumber \\
     &= \|H^{+}\|\|H(\proj_{\mathcal{Z}}y - \proj_{\mathcal{A}(x, H)}\proj_{\mathcal{Z}}y
       + \proj_{\mathcal{A}(x, H)} \proj_{\mathcal{Z}}y - x) +  G(x)\|.
       \nonumber \\
     &= \|H^{+}\|\|G(x) - H(x - \proj_{\mathcal{Z}}y)\|.
  \end{align}
  Clearly, $\proj_{\mathcal{Z}} y \in \mathcal{Z}$, so we can apply the definition of Newton
  differentiability from~\eqref{eq: underdetermined newton diff} to bound
  in~\eqref{eq underconsgtin khl lisdlj kjl} by
  \begin{equation}\label{eq:underdermined newton eq 1}
    \|\proj_{\mathcal{Z}}y - \proj_{\mathcal{A}(x, H)}\proj_{\mathcal{Z}}y\| \le c\|H^{+}\|\|(x - \proj_{\mathcal{Z}}y)\|.
  \end{equation}
  Finally, reusing~\eqref{eq: underdetemined Newton} and the bound
  $\|H^{+}\| \le \Omega$ we conclude that
  \begin{equation}\label{eq:underdetermined work in progress}
    \dist(y, \mathcal{Z}) \le cP\Omega\|(x - \proj_{\mathcal{Z}}y)\|.
  \end{equation}

  Next, we use the triangle inequality and the definition of the projector,
  together with the fact that the projection operator onto $\mathcal{Z}$ is Lipschitz
  continuous to conclude that
  \begin{align*}
    \|x - \proj_{\mathcal{Z}}y\|
    &\le \|x - \proj_{\mathcal{Z}}x\| + \|\proj_{\mathcal{Z}}x - \proj_{\mathcal{Z}}y\| \nonumber \\
    &= \dist(x, \mathcal{Z})  + \|\proj_{\mathcal{Z}}x - \proj_{\mathcal{Z}}y\| \nonumber \\
    &\le \dist(x, \mathcal{Z})  + L_{\mathcal{Z}}\|x - y\|.
  \end{align*}
  Because $y = \proj_{\mathcal{A}(x, H)}x$, we can use Lemma~\ref{eq:unconstrained basic lemma}
  to show
  \begin{equation*}
    H^{+}H(y - x) = y - x.
  \end{equation*}
  And because $y \in \mathcal{A}(x, H)$, so
  \begin{equation*}
    G(x) + H(y - x) = 0,
  \end{equation*}
  we can combine the two equations, giving
  \begin{equation*}
    x - y = H^{+}G(x).
  \end{equation*}
  From the definition of $\mathcal{Z}$, $G(\proj_{\mathcal{Z}}x) = 0$ and recalling the fact
  that $G$ is Lipschitz continuous with constant $L$ and $\|H^{+}\| \le \Omega$, we obtain
  \begin{align}
    \|x - y\|
    &\le \|H^{+}\|\|G(x) - G(\proj_{\mathcal{Z}}x)\| \le
    L\|H^{+}\|\|x - \proj_{\mathcal{Z}}x\| \nonumber \\
    &= L\|H^{+}\| \dist(x, \mathcal{Z}) \le L\Omega\dist(x, \mathcal{Z}) \label{eq:important from the past}
  \end{align}
  and
  \begin{equation}\label{eq:underdetermined here i am}
    L_{\mathcal{Z}}\|x - y\| \le L_{\mathcal{Z}}L\Omega\dist(x, \mathcal{Z}).
  \end{equation}

  Combining~\eqref{eq:underdetermined work in progress}
  with~\eqref{eq:underdermined newton eq 1} and~\eqref{eq:underdetermined here i am}
  yields
  \begin{equation*}
    \dist(y, \mathcal{Z}) \le cP\Omega(1 + L_{\mathcal{Z}}L\Omega)\dist(x, \mathcal{Z}).
  \end{equation*}

  This shows that $\mathcal{N}_{\mathcal{H}G}$ is a self mapping on a neighborhood of $\mathcal{Z}$ and
  any sequence ${\{x^k\}}_{k \in \bbN}$ satisfies
  \begin{equation}\label{eq:linear convergence for underdetermined}
    \forall k \in \bbN,\quad\dist(x^{k+1}, \mathcal{Z}) \le cP\Omega(1 + L\Omega)\dist(x^k, \mathcal{Z}).
  \end{equation}

  We are all set up to show the conclusion of the theorem. Let ${\{c^k\}}_{k \in \bbN}$
  be defined by
  \begin{equation*}
     \forall k \in \bbN,\quad c^k = \sup_{\xbar \in \mathcal{Z}}\sup_{H \in \mathcal{H}G(x^k)}\frac{\|G(x^k) - H(x^k - \xbar)\|}{\|x^k - \xbar\|},
  \end{equation*}
  and using the same reasoning as before, we know that
  \begin{equation*}
    \forall k \in \bbN,\quad\dist(x^{k+1}, \mathcal{Z}) \le c^k P\Omega(1 + L\Omega)\dist(x^k, \mathcal{Z}).
  \end{equation*}
  It remains to show that ${\{c^k\}}_{k \in \bbN}$ converges to $0$. Let $\varepsilon > 0$ and by uniform
  Newton differentiability, we know that there exists $\delta > 0$ such that, if for any $k \in \bbN$,
  $\dist(x^k, \mathcal{Z}) \le \delta$, then $c^k < \varepsilon$.
  Iterating~\eqref{eq:linear convergence for underdetermined}, we deduce that
  \begin{equation*}
    \forall k \in \bbN,\quad\dist(x^{k}, \mathcal{Z}) \le {(cP\Omega(1 + L\Omega))}^k\dist(x^0, \mathcal{Z}),
  \end{equation*}
  and solving for $k$ we know that if
  \begin{equation*}
    k \ge \frac{\log\frac{\delta}{\dist(x^0, \mathcal{Z})}}{\log cP\Omega(1 + L\Omega)},
  \end{equation*}
  then
  \begin{equation*}
    \dist(x^k, \mathcal{Z}) \le \delta,
  \end{equation*}
  and thus $c^k \le \varepsilon$. Since $\varepsilon$ was arbitrary, we conclude that
  $\lim_{k \to \infty}c^k = 0$ in order to complete the proof.
\end{proof}
\begin{remark}
  If we further assume that $\mathcal{Z}$ is compact, we can conclude that the
  sequence ${\{x^k\}}_{k \in \bbN}$ is bounded and thus it has a convergent subsequence.
\end{remark}
\begin{remark}
  This proof can be adapted to obtain linear convergence under the weaker
  assumption of uniform weak Newton differentiability, together with $M\Omega < 1$,
  where $M$ is as in Definition~\ref{def:unif-newton-diff}.
\end{remark}

Different choices of $x^{k+1} \in \mathcal{A}(H, x^k)$ can yield different convergence
results. For instance, fixing a point $\xbar \in \mathcal{Z}$ and picking
$x^{k+1} \in \proj_{\mathcal{A}(H^k, x^k)}\xbar$ can produce a superlinearly convergent
sequence, but such a method is not applicable in a practical algorithm due
to employing a point $\xbar$ in the solution set $\mathcal{Z}$. A practical approach to
approximating this algorithm has been presented in~\cite{ChaDav23Asup} and
is based on using the history of $x^{0},\dots,x^{k}$ to compute $x^k$. The choice
presented in our work has the benefit of ease of computation, as the next
lemma will show.
\begin{lemma}[Computation of Under-determined Newton Algorithm]
  Let $G:\bbRm \to \bbRn$ be uniform weakly Newton differentiable on $\mathcal{Z} = \{\xbar \in \bbRm~|~G(\xbar) = 0\}$
  and with $n < m$. Then $\mathcal{N}_{\mathcal{H}G}$ defined
  in Definition~\ref{def:underdetemined newton method} can be computed by
  \begin{equation*}
    \mathcal{N}_{\mathcal{H}G}(x) = \{x - H^{+}G(x)~|~H \in \mathcal{H}G(x)\}.
  \end{equation*}
\end{lemma}
\begin{proof}
  Let $x \in \bbRm$ and $y = \proj_{\mathcal{A}(x, H)}$. We rephrase the projection as an optimization problem
  \begin{equation*}
    \minimize{y \in \bbRm}{\|x - y\|^2}{y \in \mathcal{A}(x, H),}
  \end{equation*}
  or alternatively, using the definition of $\mathcal{A}(x, H)$
  \begin{equation}\label{eq:uunderdetermined variational}
    \minimize{y \in \bbRm}{\|x - y\|^2}{G(x) + H(y - x) = 0.}
  \end{equation}
  Because this is a convex problem, the first order optimality conditions are
  necessary and sufficient, so $y$ is the solution
  of~\eqref{eq:uunderdetermined variational} if and only if
  \begin{equation*}
    2(x - y) \perp \ker H.
  \end{equation*}
  Equivalently, there exists $z \in \bbRn$ with
  \begin{equation}\label{eq:underdeermined lagrange multipliers}
    x - y = H^T z.
  \end{equation}
  Because $y$ is the solution to the optimization problem~\eqref{eq:uunderdetermined variational},
  it is a feasible point,
  so
  \begin{equation*}
    G(x) + H(y - x) = 0.
  \end{equation*}
  Substituting from~\eqref{eq:underdeermined lagrange multipliers} shows that
  \begin{equation*}
    G(x) + H(x - H^{T}z - x) = 0.
  \end{equation*}
  Next, because $H$ has full rank, we know that $HH^T$ is invertible and this
  allows us to solve for $z$ yielding,
  \begin{equation*}
    z = {(HH^T)}^{-1}G(x)
  \end{equation*}
  The final step requires us to use~\eqref{eq:underdeermined lagrange multipliers}
  together with the definition of the pseudo inverse matrix to express
  \begin{equation*}
    y = x - H^{T}{(HH^T)}^{-1}G(x) = x - H^{+}G(x),
  \end{equation*}
  in order to complete the proof.
\end{proof}

\begin{remark}
  Using the same assumptions and arguments, the update rule
  (without the line-search condition) from Algorithm 1 from~\cite{PolTre17Solv}
  can be expressed as
  \begin{equation*}
    \mathcal{N}_{\mathcal{H}G}:\bbRm \setto \bbRm,\quad \mathcal{N}_{\mathcal{H}G}(x) = \{- H^{+}(G(x) - Hx)~|~H \in \mathcal{H}G(x)\}.
  \end{equation*}
\end{remark}

\section{Application: (Semi)-Smooth Systems}
In this section we focus on particularizing the proposed algorithm to smooth and
semi-smooth systems. The main theoretical result is that for single-valued
Newton differentials linear geometric compatibility is implied by uniform
continuity. This fact can be used, together with the Newton differentiability
of smooth and semi-smooth maps to state corollaries from
Theorem~\ref{thm:superlinear convergence of underdetemined newton}.
\begin{theorem}
  Let $G: \bbRm \to \bbRn$ be in Lipschitz continuous with constant $L$ and
  Newton differentiable with a single-valued Newton differential $\mathcal{H} G$.
  Denote
  $\mathcal{Z} = \{\xbar \in \bbRm~|~G(\xbar) = 0\}$.
  Assume that $\proj_\mathcal{Z}$ exists and is single-valued and there exits $\Omega \in (0, \infty)$ such
  that $\|{\mathcal{H} G(x)}^+\| \le \Omega$ for all $x \in \bbRm$.
  If $\mathcal{H} G$ is uniformly continuous then it is
  {\em geometrically compatible with $\mathcal{Z}$}.
\end{theorem}
\begin{proof}
  Let us simplify the notation by defining for $x \in \bbRm$ $y = \proj_{\mathcal{A}(x, \mathcal{H} G(x))}x$,
  $\ybar = \proj_\mathcal{Z} y$, and $z = \proj_{\mathcal{A}(x, \mathcal{H} G(x))}x$.

  Clearly, setting
  \begin{equation*}
    P = \frac{\| y - \ybar \| \| \ybar - z\|}{\langle y - \ybar, \ybar - z \rangle}
  \end{equation*}
  would give equality in~\eqref{eq:linearly geometrically compatible}, so it
  only remains to show that this choice of $P$ is bounded, or equivalently that
  there is $\Omega_P > 0$ such that
  \begin{equation*}
    \cos  \angle(y - \ybar, \ybar - z) \ge \Omega_P.
  \end{equation*}

  Because
  \begin{equation*}
    y - \ybar \perp \ker \mathcal{H} G(\ybar)
  \end{equation*}
  and
  \begin{equation}
     \ybar - z \perp \ker \mathcal{H} G(x)
  \end{equation}
  we can compute
  \begin{align*}
    \cos  \angle(y - \ybar, \ybar - z) \ge \Omega_P
    &= \cos \arccos \frac{\tr {\mathcal{H} G(x)}\mathcal{H} G(\ybar)}{\sqrt{\tr {\mathcal{H} G(x)}^2
      \tr {\mathcal{H} G(\ybar)}^2}} \\
    &= \frac{\tr {\mathcal{H} G(x)}\mathcal{H} G(\ybar)}{\sqrt{\tr {\mathcal{H} G(x)}^2 \tr {\mathcal{H} G(\ybar)}^2}}. \\
  \end{align*}

  Next, we are going to compute
  \begin{align}
    &\left|\frac{\tr {\mathcal{H} G(\ybar)} {\mathcal{H} G(\ybar)}}{\sqrt{\tr {\mathcal{H} G(\ybar)}^2 \tr {\mathcal{H} G(\ybar)}^2}}
    - \frac{\tr {\mathcal{H} G(x)}{\mathcal{H} G(\ybar)}}{\sqrt{\tr {\mathcal{H} G(x)}^2 \tr {\mathcal{H} G(\ybar)}^2}}\right|
      \nonumber \\
    \quad&=\left|\frac{\tr {\mathcal{H} G(\ybar)} {\mathcal{H} G(\ybar)}}{\sqrt{\tr {\mathcal{H} G(\ybar)}^2 \tr {\mathcal{H} G(\ybar)}^2}}
      - \frac{\tr {\mathcal{H} G(x)}{\mathcal{H} G(\ybar)}}{\sqrt{\tr {\mathcal{H} G(x)}^2 \tr {\mathcal{H} G(\ybar)}^2}}\right|
      \nonumber \\
    &= \frac{\tr ({\mathcal{H} G(\ybar)} \sqrt{\tr {\mathcal{H} G(x)}^2} - {\mathcal{H} G(x)} \sqrt{\tr {\mathcal{H} G(\ybar)}^2})
      {\mathcal{H} G(\ybar)} \sqrt{\tr {\mathcal{H} G(\ybar)}^2}
      }{\sqrt{\tr {\mathcal{H} G(\ybar)}^2 \tr {\mathcal{H} G(\ybar)}^2 \tr {\mathcal{H} G(x)}^2 \tr {\mathcal{H} G(\ybar)}^2}}.
      \label{eq:proof smooth I am close}
  \end{align}

  Because the Frobenius norm is induced by the trace inner product, so
  $\|{\mathcal{H} G(x)}\|_F = \sqrt{\tr {\mathcal{H} G(x)}^2}$, we can use the norm equivalence to conclude
  that there exists a constant $c_F > 0$ such that
  $\|{\mathcal{H} G(x)}^{+}\|_F \le c_F\|{\mathcal{H} G(x)}^{+}\| \le c_F\Omega_G$, and also
  $\|{\mathcal{H} G(x)}{\mathcal{H} G(x)}^{+}{\mathcal{H} G(x)}\|_F = \|{\mathcal{H} G(x)}\|_F \le \|{\mathcal{H} G(x)}\|_F^2\|{\mathcal{H} G(x)}^{+}\|_F$
  holds for any ${\mathcal{H} G(x)}$ with $\|{\mathcal{H} G(x)}\| \le \Omega_G$, so
  \begin{equation*}
    \frac{1}{\|{\mathcal{H} G(x)}\|_F} \le c_F\Omega_G.
  \end{equation*}

  Substituting in~\eqref{eq:proof smooth I am close}, gives
  \begin{align}
    1
    &- \frac{\tr {\mathcal{H} G(x)}{\mathcal{H} G(\ybar)}}{\sqrt{\tr {\mathcal{H} G(x)}^2 \tr {\mathcal{H} G(\ybar)}^2}}
      \nonumber \\
    &\le c_F^2\Omega_G^2\tr (({\mathcal{H} G(\ybar)} \|{\mathcal{H} G(x)}\|_F
      - {\mathcal{H} G(x)} \|{\mathcal{H} G(\ybar)}^2\|_F) {\mathcal{H} G(\ybar)}) \|{\mathcal{H} G(\ybar)}\|_F
    \nonumber \\
    &\le c_F^3\Omega_G^3\tr (({\mathcal{H} G(\ybar)} \|{\mathcal{H} G(x)}\|_F
      - {\mathcal{H} G(x)} \|{\mathcal{H} G(\ybar)}\|_F) {\mathcal{H} G(\ybar)}) \nonumber \\
    &\le c_F^3\Omega_G^3 \| {\mathcal{H} G(\ybar)}\|{\mathcal{H} G(x)}\|_F
      - {\mathcal{H} G(x)} \|{\mathcal{H} G(\ybar)}\|_F\|_F \|{\mathcal{H} G(\ybar)}\|_F \nonumber \\
    &\le c_F^4\Omega_G^4 \| {\mathcal{H} G(\ybar)}\|{\mathcal{H} G(x)}\|_F
      - {\mathcal{H} G(x)} \|{\mathcal{H} G(\ybar)}\|_F\|_F \nonumber \\
    &\le c_F^4\Omega_G^4 \| {\mathcal{H} G(\ybar)}(\|{\mathcal{H} G(x)}\|_F -\|{\mathcal{H} G(\ybar)}\|_F)
      + ({\mathcal{H} G(\ybar)} - {\mathcal{H} G(x)})\|{\mathcal{H} G(\ybar)}\|_F\|_F \nonumber \\
    &\le c_F^4\Omega_G^4 \| {\mathcal{H} G(\ybar)}(\|{\mathcal{H} G(x)}\|_F -\|{\mathcal{H} G(\ybar)}\|_F)
      + ({\mathcal{H} G(\ybar)} - {\mathcal{H} G(x)})\|{\mathcal{H} G(\ybar)}\|_F\|_F \nonumber \\
    &\le c_F^4\Omega_G^4 \|{\mathcal{H} G(\ybar)}\|_F|\|{\mathcal{H} G(x)}\|_F
      - \|{\mathcal{H} G(\ybar)}\|_F| + \|{\mathcal{H} G(\ybar)} - {\mathcal{H} G(x)}\|_F\|{\mathcal{H} G(\ybar)}\|_F \nonumber \\
    &\le 2 c_F^5\Omega_G^5 \|{\mathcal{H} G(\ybar)} - {\mathcal{H} G(x)}\|_F \nonumber \\
    &\le 2 c_F^6\Omega_G^5 \|{\mathcal{H} G(\ybar)} - {\mathcal{H} G(x)}\|.\label{eq:proof-main inequakity big}
  \end{align}

  Finally, we can use the uniform continuity of $\mathcal{H} G$ to conclude that there
  exists $\delta_G>0$, such that for all $x \in \bbRm$ with $\|x - \ybar\| \le \delta_G$,
  \begin{equation*}
    \|{\mathcal{H} G(\ybar)} - {\mathcal{H} G(x)}\| \le \frac{1}{4 c_F^6\Omega_G^5 }.
  \end{equation*}
  This allows us to rearrange in~\eqref{eq:proof-main inequakity big}, yielding
  \begin{equation*}
    \frac{\tr {\mathcal{H} G(x)}{\mathcal{H} G(\ybar)}}{\sqrt{\tr {\mathcal{H} G(x)}^2 \tr {\mathcal{H} G(\ybar)}^2}} \ge 1
    - 2 c_F^6\Omega_G^5 \frac{1}{4 c_F^6\Omega_G^5 } = \frac{1}{2},
  \end{equation*}
  for any $x$ with $\|x - \ybar\| \le \delta_G$.

  The final step of the proof consists in showing that there is a $\delta > 0$ such the
  $\dist(x, \mathcal{Z}) \le \delta$ implies that $\|x - \ybar\| \le \delta_G$. Let $\xbar = \proj_\mathcal{Z} x $ and
  because $\ybar = \proj_{\mathcal{Z}}y$, we can see that $\|y - \ybar\| \le \|y - \xbar\|$, so using the
  triangle inequality we can bound
  \begin{equation*}
    \|x - \ybar\| \le \|x - y\| + \|y - x\| + \|x - \xbar\|.
  \end{equation*}
  Using~\ref{eq:important from the past} we conclude that
  \begin{equation*}
    \|x - \ybar\| \le (2 L\Omega + 1)\| x - \xbar \|
  \end{equation*}
  and as such taking $\delta = \delta_G/(2 L\Omega + 1)$ completes the proof.
\end{proof}

This allows us to state the result concerning the application of our algorithm
to smooth problems.
\begin{corollary}
  Let $G: \bbRm \to \bbRn$ be Lipschitz continuous with constant $L$ and $\mathcal{C}^1$.
  Denote the feasibility set $\mathcal{Z} = \{\xbar \in \bbRm~|~G(\xbar) = 0\}$.
  Assume that the projection operator onto $\mathcal{Z}$ is Lipschitz continuous with
  constant $L_{\mathcal{Z}}$ and that $\mathcal{Z}$ is compact.
  Furthermore, assume that there is $\Omega > 0$ such that $\nabla G(x)$ has full rank
  and $\|{\nabla G(x)}^+\| \le \Omega$.
  Then any sequence
  ${\{x^k\}}_{k \in \mathbb{N}}$ with $x^0$ near $\mathcal{Z}$ and generated
  by~\eqref{eq:newton underdetermined iteration} with $\mathcal{H} G = \nabla G$, satisfies
  \begin{equation*}
    \forall k \in \bbN,\quad \dist(x^{k+1}, \mathcal{Z}) \le c^k\dist(x^k, \mathcal{Z}),
  \end{equation*}
  where ${\{c^k\}}_{k \in \bbN}$ is a positive real sequence convergent to $0$.
\end{corollary}

\begin{corollary}
  Let $G: \bbRm \to \bbRn$ be Lipschitz continuous with constant $L$ and uniformly
  semi-smooth* on
  the feasibility set $\mathcal{Z} = \{\xbar \in \bbRm~|~G(\xbar) = 0\}$.
  Assume that the projection operator onto $\mathcal{Z}$ is Lipschitz continuous with
  constant $L_{\mathcal{Z}}$ and that $\partial^C G$ is single valued.
  Furthermore, assume that there is $\Omega > 0$ such that $\partial^C G(x)$ has full rank
  and $\|{\partial^C G(x)}^+\| \le \Omega$.
  Then any sequence
  ${\{x^k\}}_{k \in \mathbb{N}}$ with $x^0$ near $\mathcal{Z}$ and generated
  by~\eqref{eq:newton underdetermined iteration} with $\mathcal{H} G = \partial^C G$, satisfies
  \begin{equation*}
    \forall k \in \bbN,\quad \dist(x^{k+1}, \mathcal{Z}) \le c^k\dist(x^k, \mathcal{Z}),
  \end{equation*}
  where ${\{c^k\}}_{k \in \bbN}$ is a positive real sequence convergent to $0$.
\end{corollary}

\subsection{Numerical Experiments}
Our method resembles~\cite[Algorithm 1]{PolTre17Solv} and as such a
direct comparison is necessary. The key difference between that method
and our work is that we do not employ a globalization strategy, but can
handle non-smooth objectives. The test numerical problem from~\cite{PolTre17Solv}
is finding the roots of
\begin{equation*}
  G(x)_i = \varphi((Cx)_i - b_i) - y_i,
\end{equation*}
where $C \in \bbRnxm$, $b,y \in \bbRn$ are random, $c_i$ is the $i$-th component of the vector $c$
and $\varphi(t) = t/(1 + e^{-|t|})$.
As expected, because of the smoothness of the problem, both algorithms behave
nearly identically, attaining quadratic convergence.

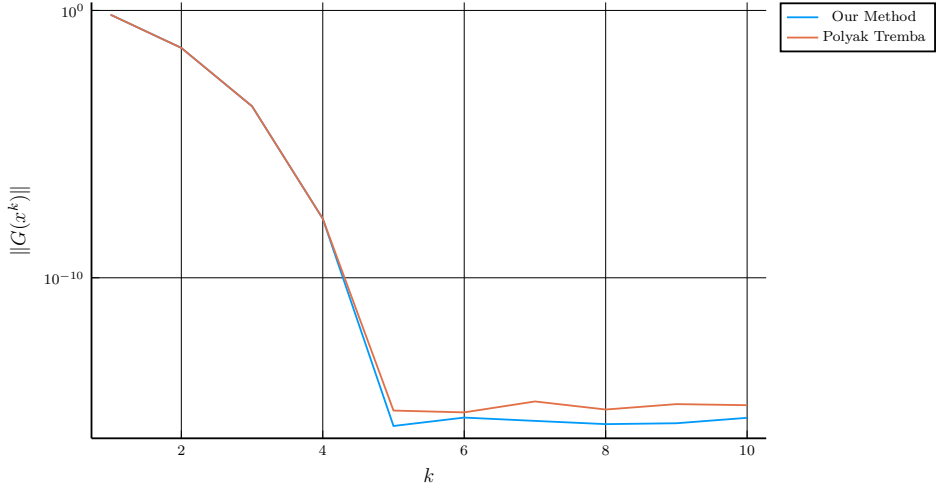
\begin{figure}[ht!]
  \begin{adjustbox}{width=\textwidth}%
    \begin{tikzpicture}[/tikz/background rectangle/.style={fill={rgb,1:red,1.0;green,1.0;blue,1.0}, fill opacity={1.0}, draw opacity={1.0}}, show background rectangle]
      \begin{axis}[point meta max={nan}, point meta min={nan}, legend cell align={left}, legend columns={1}, title={}, title style={at={{(0.5,1)}}, anchor={south}, font={{\fontsize{14 pt}{18.2 pt}\selectfont}}, color={rgb,1:red,0.0;green,0.0;blue,0.0}, draw opacity={1.0}, rotate={0.0}, align={center}}, legend style={color={rgb,1:red,0.0;green,0.0;blue,0.0}, draw opacity={1.0}, line width={1}, solid, fill={rgb,1:red,1.0;green,1.0;blue,1.0}, fill opacity={1.0}, text opacity={1.0}, font={{\fontsize{8 pt}{10.4 pt}\selectfont}}, text={rgb,1:red,0.0;green,0.0;blue,0.0}, cells={anchor={center}}, at={(1.02, 1)}, anchor={north west}}, axis background/.style={fill={rgb,1:red,1.0;green,1.0;blue,1.0}, opacity={1.0}}, anchor={north west}, xshift={1.0mm}, yshift={-1.0mm}, width={145.4mm}, height={99.6mm}, scaled x ticks={false}, xlabel={$k$}, x tick style={color={rgb,1:red,0.0;green,0.0;blue,0.0}, opacity={1.0}}, x tick label style={color={rgb,1:red,0.0;green,0.0;blue,0.0}, opacity={1.0}, rotate={0}}, xlabel style={at={(ticklabel cs:0.5)}, anchor=near ticklabel, at={{(ticklabel cs:0.5)}}, anchor={near ticklabel}, font={{\fontsize{11 pt}{14.3 pt}\selectfont}}, color={rgb,1:red,0.0;green,0.0;blue,0.0}, draw opacity={1.0}, rotate={0.0}}, xmajorgrids={true}, xmin={0.7299999999999995}, xmax={10.27}, xticklabels={{$2$,$4$,$6$,$8$,$10$}}, xtick={{2.0,4.0,6.0,8.0,10.0}}, xtick align={inside}, xticklabel style={font={{\fontsize{8 pt}{10.4 pt}\selectfont}}, color={rgb,1:red,0.0;green,0.0;blue,0.0}, draw opacity={1.0}, rotate={0.0}}, x grid style={color={rgb,1:red,0.0;green,0.0;blue,0.0}, draw opacity={0.1}, line width={0.5}, solid}, axis x line*={left}, x axis line style={color={rgb,1:red,0.0;green,0.0;blue,0.0}, draw opacity={1.0}, line width={1}, solid}, scaled y ticks={false}, ylabel={$\|G(x^k)\|$}, y tick style={color={rgb,1:red,0.0;green,0.0;blue,0.0}, opacity={1.0}}, y tick label style={color={rgb,1:red,0.0;green,0.0;blue,0.0}, opacity={1.0}, rotate={0}}, ylabel style={at={(ticklabel cs:0.5)}, anchor=near ticklabel, at={{(ticklabel cs:0.5)}}, anchor={near ticklabel}, font={{\fontsize{11 pt}{14.3 pt}\selectfont}}, color={rgb,1:red,0.0;green,0.0;blue,0.0}, draw opacity={1.0}, rotate={0.0}}, ymode={log}, log basis y={10}, ymajorgrids={true}, ymin={9.890925927657651e-17}, ymax={1.9970627596079846}, yticklabels={{$10^{-10}$,$10^{0}$}}, ytick={{1.0e-10,1.0}}, ytick align={inside}, yticklabel style={font={{\fontsize{8 pt}{10.4 pt}\selectfont}}, color={rgb,1:red,0.0;green,0.0;blue,0.0}, draw opacity={1.0}, rotate={0.0}}, y grid style={color={rgb,1:red,0.0;green,0.0;blue,0.0}, draw opacity={0.1}, line width={0.5}, solid}, axis y line*={left}, y axis line style={color={rgb,1:red,0.0;green,0.0;blue,0.0}, draw opacity={1.0}, line width={1}, solid}, colorbar={false}]
        \addplot[color={rgb,1:red,0.0;green,0.6056;blue,0.9787}, name path={1}, draw opacity={1.0}, line width={1}, solid]
        table[row sep={\\}]
        {
          \\
          1.0  0.6901207235042334  \\
          2.0  0.03969684547634147  \\
          3.0  0.0002606393409578396  \\
          4.0  1.6260250118805842e-8  \\
          5.0  2.8622238335152586e-16  \\
          6.0  5.902734891075384e-16  \\
          7.0  4.43335039951723e-16  \\
          8.0  3.355394223065532e-16  \\
          9.0  3.61868560698565e-16  \\
          10.0  5.775375822671982e-16  \\
        }
        ;
        \addlegendentry {Our Method}
        \addplot[color={rgb,1:red,0.8889;green,0.4356;blue,0.2781}, name path={2}, draw opacity={1.0}, line width={1}, solid]
        table[row sep={\\}]
        {
          \\
          1.0  0.6901207235042334  \\
          2.0  0.039696845476343195  \\
          3.0  0.0002606393409583759  \\
          4.0  1.626024988706782e-8  \\
          5.0  1.079736436130962e-15  \\
          6.0  9.234292161343797e-16  \\
          7.0  2.3651450288898032e-15  \\
          8.0  1.1808872524523116e-15  \\
          9.0  1.8813912454502658e-15  \\
          10.0  1.7149632713319975e-15  \\
        }
        ;
        \addlegendentry {Polyak Tremba}
      \end{axis}
    \end{tikzpicture}
  \end{adjustbox}
  \caption{Objective value, $\|G(x^k)\|$ over $k$ for our Algorithm
    and~\cite[Algorithm 1]{PolTre17Solv}, observing nearly identical behavior due to the smoothness
  of the problem}
\end{figure}

Another Newton-type method for under-determined systems is presented
in~\cite{Kub11Inte}. A direct comparison between our method and their method
is impossible due to the fact that the latter employs interval arithmetic in
order to approximate the entire set of solutions and not just to find one
possible solution. Nonetheless, the problems tackled there can serve as an
useful testing ground for our method. These problems come from interesting
real life applications, such as inverse kinematics of a robot. Consider the
equation systems $G(x) = 0$, where
\begin{equation}\label{eq:inter_p1}\tag{$\mathcal{P}_1$}
  G(x) = \begin{bmatrix}(x_{1}^2 + x_{2}^2 - 4)  (x_{1}^2 + x_{2}^2 - 1)\end{bmatrix}
\end{equation}
\begin{equation}\label{eq:inter_p2}\tag{$\mathcal{P}_2$}
  G(x) = \begin{bmatrix}x_{1}^2 + x_{2}^2 - x_{3}\\  x_{1}^2 + x_{2}^2 - 1.1x_{3}\end{bmatrix}
\end{equation}
\begin{equation}\label{eq:inter_p3}\tag{$\mathcal{P}_3$}
  G(x) = \begin{bmatrix}
    x_{1}^2 + x_{2}^2 - 1\\
    x_{3}^2 + x_{4}^2 - 1\\
    x_{5}^2 + x_{6}^2 - 1\\
    x_{7}^2 + x_{8}^2 - 1\\
    0.004731x_{1}x_{2} - 0.3578x_{2}x_{3} - 0.1238x_{1} - 0.001637x_{2} - 0.9338x_{4} + x_{7}\\
    0.2238x_{1}x_{3} + 0.7623x_{2}x_{3} + 0.2638x_{1} - 0.07745x_{2} - 0.6734x_{4} - 0.6022\\
    x_{6}x_{8} + 0.3578x_{1} + 0.004731x_{2}\\
  \end{bmatrix}
\end{equation}
\begin{equation}\label{eq:inter_p4}\tag{$\mathcal{P}_4$}
  G(x) = \begin{bmatrix}
    -3.933x_{1} + 0.107x_{2} + 0.126x_{3} - 9.99x_{5} - 45.83x_{7} + \\
    - 7.64x_{8} - 0.727x_{2}x_{3} + 8.39x_{3}x_{4} - 684.4x_{4}x_{5} + 63.5x_{4}x_{7}\\
    -0.987x_{2} - 22.95x_{4} - 28.37x_{6} + 0.949x_{1}x_{3} + 0.173x_{1}x_{5}\\
    0.002x_{1} - 0.235x_{3} + 5.67x_{5} + 0.921x_{7} - 6.51x_{8} - 0.716x_{1}x_{2} + \\
    - 1.578x_{1}x_{4} + 1.132x_{4}x_{7}\\
    x_{1} - x_{4} - 0.168x_{6} - x_{1}x_{2}\\
    -x_{3} - 0.196x_{5} - 0.0071x_{7} + x_{1}x_{4}\\
  \end{bmatrix}
\end{equation}
To treat all the examples from~\cite{Kub11Inte}, we also add problem~($\mathcal{P}_{3b}$)
by removing the last equation from~\eqref{eq:inter_p3} and~($\mathcal{P}_{4b}$)
by setting $x_6 = 0.1$ and $x_8 = 0$ in~\eqref{eq:inter_p4}.

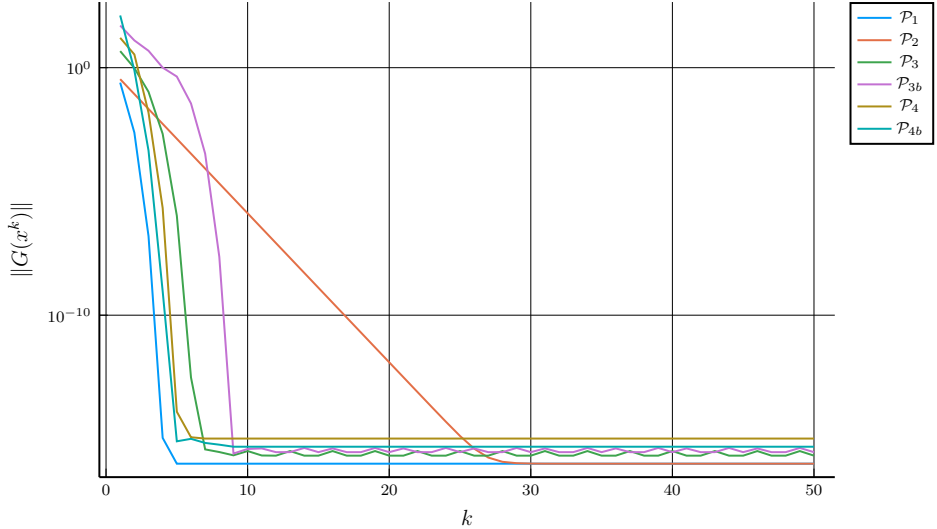
\begin{figure}[ht!]
  \begin{adjustbox}{width=\textwidth}%
    \begin{tikzpicture}[/tikz/background rectangle/.style={fill={rgb,1:red,1.0;green,1.0;blue,1.0}, fill opacity={1.0}, draw opacity={1.0}}, show background rectangle]
      \begin{axis}[point meta max={nan}, point meta min={nan}, legend cell align={left}, legend columns={1}, title={}, title style={at={{(0.5,1)}}, anchor={south}, font={{\fontsize{14 pt}{18.2 pt}\selectfont}}, color={rgb,1:red,0.0;green,0.0;blue,0.0}, draw opacity={1.0}, rotate={0.0}, align={center}}, legend style={color={rgb,1:red,0.0;green,0.0;blue,0.0}, draw opacity={1.0}, line width={1}, solid, fill={rgb,1:red,1.0;green,1.0;blue,1.0}, fill opacity={1.0}, text opacity={1.0}, font={{\fontsize{8 pt}{10.4 pt}\selectfont}}, text={rgb,1:red,0.0;green,0.0;blue,0.0}, cells={anchor={center}}, at={(1.02, 1)}, anchor={north west}}, axis background/.style={fill={rgb,1:red,1.0;green,1.0;blue,1.0}, opacity={1.0}}, anchor={north west}, xshift={1.0mm}, yshift={-1.0mm}, width={145.4mm}, height={99.6mm}, scaled x ticks={false}, xlabel={$k$}, x tick style={color={rgb,1:red,0.0;green,0.0;blue,0.0}, opacity={1.0}}, x tick label style={color={rgb,1:red,0.0;green,0.0;blue,0.0}, opacity={1.0}, rotate={0}}, xlabel style={at={(ticklabel cs:0.5)}, anchor=near ticklabel, at={{(ticklabel cs:0.5)}}, anchor={near ticklabel}, font={{\fontsize{11 pt}{14.3 pt}\selectfont}}, color={rgb,1:red,0.0;green,0.0;blue,0.0}, draw opacity={1.0}, rotate={0.0}}, xmajorgrids={true}, xmin={-0.4700000000000024}, xmax={51.47}, xticklabels={{$0$,$10$,$20$,$30$,$40$,$50$}}, xtick={{0.0,10.0,20.0,30.0,40.0,50.0}}, xtick align={inside}, xticklabel style={font={{\fontsize{8 pt}{10.4 pt}\selectfont}}, color={rgb,1:red,0.0;green,0.0;blue,0.0}, draw opacity={1.0}, rotate={0.0}}, x grid style={color={rgb,1:red,0.0;green,0.0;blue,0.0}, draw opacity={0.1}, line width={0.5}, solid}, axis x line*={left}, x axis line style={color={rgb,1:red,0.0;green,0.0;blue,0.0}, draw opacity={1.0}, line width={1}, solid}, scaled y ticks={false}, ylabel={$\|G(x^k)\|$}, y tick style={color={rgb,1:red,0.0;green,0.0;blue,0.0}, opacity={1.0}}, y tick label style={color={rgb,1:red,0.0;green,0.0;blue,0.0}, opacity={1.0}, rotate={0}}, ylabel style={at={(ticklabel cs:0.5)}, anchor=near ticklabel, at={{(ticklabel cs:0.5)}}, anchor={near ticklabel}, font={{\fontsize{11 pt}{14.3 pt}\selectfont}}, color={rgb,1:red,0.0;green,0.0;blue,0.0}, draw opacity={1.0}, rotate={0.0}}, ymode={log}, log basis y={10}, ymajorgrids={true}, ymin={2.8622350426037563e-17}, ymax={449.9023082782237}, yticklabels={{$10^{-10}$,$10^{0}$}}, ytick={{1.0e-10,1.0}}, ytick align={inside}, yticklabel style={font={{\fontsize{8 pt}{10.4 pt}\selectfont}}, color={rgb,1:red,0.0;green,0.0;blue,0.0}, draw opacity={1.0}, rotate={0.0}}, y grid style={color={rgb,1:red,0.0;green,0.0;blue,0.0}, draw opacity={0.1}, line width={0.5}, solid}, axis y line*={left}, y axis line style={color={rgb,1:red,0.0;green,0.0;blue,0.0}, draw opacity={1.0}, line width={1}, solid}, colorbar={false}]
        \addplot[color={rgb,1:red,0.0;green,0.6056;blue,0.9787}, name path={1}, draw opacity={1.0}, line width={1}, solid]
        table[row sep={\\}]
        {
          \\
          1.0  0.24894843978818162  \\
          2.0  0.002386675527274292  \\
          3.0  1.5768319170212898e-7  \\
          4.0  1.0992007221626411e-15  \\
          5.0  1.0e-16  \\
          6.0  1.0e-16  \\
          7.0  1.0e-16  \\
          8.0  1.0e-16  \\
          9.0  1.0e-16  \\
          10.0  1.0e-16  \\
          11.0  1.0e-16  \\
          12.0  1.0e-16  \\
          13.0  1.0e-16  \\
          14.0  1.0e-16  \\
          15.0  1.0e-16  \\
          16.0  1.0e-16  \\
          17.0  1.0e-16  \\
          18.0  1.0e-16  \\
          19.0  1.0e-16  \\
          20.0  1.0e-16  \\
          21.0  1.0e-16  \\
          22.0  1.0e-16  \\
          23.0  1.0e-16  \\
          24.0  1.0e-16  \\
          25.0  1.0e-16  \\
          26.0  1.0e-16  \\
          27.0  1.0e-16  \\
          28.0  1.0e-16  \\
          29.0  1.0e-16  \\
          30.0  1.0e-16  \\
          31.0  1.0e-16  \\
          32.0  1.0e-16  \\
          33.0  1.0e-16  \\
          34.0  1.0e-16  \\
          35.0  1.0e-16  \\
          36.0  1.0e-16  \\
          37.0  1.0e-16  \\
          38.0  1.0e-16  \\
          39.0  1.0e-16  \\
          40.0  1.0e-16  \\
          41.0  1.0e-16  \\
          42.0  1.0e-16  \\
          43.0  1.0e-16  \\
          44.0  1.0e-16  \\
          45.0  1.0e-16  \\
          46.0  1.0e-16  \\
          47.0  1.0e-16  \\
          48.0  1.0e-16  \\
          49.0  1.0e-16  \\
          50.0  1.0e-16  \\
        }
        ;
        \addlegendentry {$\mathcal{P}_1$}
        \addplot[color={rgb,1:red,0.8889;green,0.4356;blue,0.2781}, name path={2}, draw opacity={1.0}, line width={1}, solid]
        table[row sep={\\}]
        {
          \\
          1.0  0.343353941770593  \\
          2.0  0.08583848544264874  \\
          3.0  0.0214596213606619  \\
          4.0  0.005364905340165593  \\
          5.0  0.0013412263350414384  \\
          6.0  0.0003353065837604432  \\
          7.0  8.382664594020098e-5  \\
          8.0  2.095666148512315e-5  \\
          9.0  5.239165371352619e-6  \\
          10.0  1.3097913429128295e-6  \\
          11.0  3.2744783580347784e-7  \\
          12.0  8.186195902631083e-8  \\
          13.0  2.0465489831291e-8  \\
          14.0  5.116372532850309e-9  \\
          15.0  1.2790932082521235e-9  \\
          16.0  3.1977337705343246e-10  \\
          17.0  7.9943419271287e-11  \\
          18.0  1.9985929819400975e-11  \\
          19.0  4.996557446539602e-12  \\
          20.0  1.2492143645391258e-12  \\
          21.0  3.123785923672327e-13  \\
          22.0  7.816964731771272e-14  \\
          23.0  1.961741183315014e-14  \\
          24.0  4.979353008083941e-15  \\
          25.0  1.3198382868925896e-15  \\
          26.0  4.049595753032702e-16  \\
          27.0  1.7623987494369762e-16  \\
          28.0  1.1905997693327575e-16  \\
          29.0  1.0476499086892026e-16  \\
          30.0  1.0119124751019432e-16  \\
          31.0  1.0029781174939854e-16  \\
          32.0  1.0007445299172065e-16  \\
          33.0  1.0001861331093454e-16  \\
          34.0  1.0000465328818776e-16  \\
          35.0  1.0000116334221362e-16  \\
          36.0  1.0000029083381309e-16  \\
          37.0  1.0000007270457936e-16  \\
          38.0  1.0000001817700206e-16  \\
          39.0  1.0000000454458604e-16  \\
          40.0  1.0000000113651602e-16  \\
          41.0  1.0000000028373687e-16  \\
          42.0  1.00000000071155e-16  \\
          43.0  1.0000000001774961e-16  \\
          44.0  1.000000000044353e-16  \\
          45.0  1.0000000000109485e-16  \\
          46.0  1.0000000000027674e-16  \\
          47.0  1.0000000000007395e-16  \\
          48.0  1.0000000000001788e-16  \\
          49.0  1.0000000000000077e-16  \\
          50.0  1.0000000000000077e-16  \\
        }
        ;
        \addlegendentry {$\mathcal{P}_2$}
        \addplot[color={rgb,1:red,0.2422;green,0.6433;blue,0.3044}, name path={3}, draw opacity={1.0}, line width={1}, solid]
        table[row sep={\\}]
        {
          \\
          1.0  4.669770340201882  \\
          2.0  0.9656127261933198  \\
          3.0  0.1058772022233598  \\
          4.0  0.0021150875035540603  \\
          5.0  1.0446109376871488e-6  \\
          6.0  3.048145310504078e-13  \\
          7.0  3.7833528704240246e-16  \\
          8.0  2.9540102647498287e-16  \\
          9.0  2.1775693440128312e-16  \\
          10.0  3.242025235105681e-16  \\
          11.0  2.1443916996305593e-16  \\
          12.0  2.1102230246251565e-16  \\
          13.0  3.237726045655905e-16  \\
          14.0  2.1123893155135926e-16  \\
          15.0  2.1102230246251565e-16  \\
          16.0  3.237726045655905e-16  \\
          17.0  2.1123893155135926e-16  \\
          18.0  2.1102230246251565e-16  \\
          19.0  3.237726045655905e-16  \\
          20.0  2.1123893155135926e-16  \\
          21.0  2.1102230246251565e-16  \\
          22.0  3.237726045655905e-16  \\
          23.0  2.1123893155135926e-16  \\
          24.0  2.1102230246251565e-16  \\
          25.0  3.237726045655905e-16  \\
          26.0  2.1123893155135926e-16  \\
          27.0  2.1102230246251565e-16  \\
          28.0  3.237726045655905e-16  \\
          29.0  2.1123893155135926e-16  \\
          30.0  2.1102230246251565e-16  \\
          31.0  3.237726045655905e-16  \\
          32.0  2.1123893155135926e-16  \\
          33.0  2.1102230246251565e-16  \\
          34.0  3.237726045655905e-16  \\
          35.0  2.1123893155135926e-16  \\
          36.0  2.1102230246251565e-16  \\
          37.0  3.237726045655905e-16  \\
          38.0  2.1123893155135926e-16  \\
          39.0  2.1102230246251565e-16  \\
          40.0  3.237726045655905e-16  \\
          41.0  2.1123893155135926e-16  \\
          42.0  2.1102230246251565e-16  \\
          43.0  3.237726045655905e-16  \\
          44.0  2.1123893155135926e-16  \\
          45.0  2.1102230246251565e-16  \\
          46.0  3.237726045655905e-16  \\
          47.0  2.1123893155135926e-16  \\
          48.0  2.1102230246251565e-16  \\
          49.0  3.237726045655905e-16  \\
          50.0  2.1123893155135926e-16  \\
        }
        ;
        \addlegendentry {$\mathcal{P}_3$}
        \addplot[color={rgb,1:red,0.7644;green,0.4441;blue,0.8243}, name path={4}, draw opacity={1.0}, line width={1}, solid]
        table[row sep={\\}]
        {
          \\
          1.0  50.38247045939795  \\
          2.0  12.652870995545156  \\
          3.0  4.8186907108413815  \\
          4.0  1.0083130607193838  \\
          5.0  0.4381917579911258  \\
          6.0  0.035899939651402486  \\
          7.0  0.0003362166451939486  \\
          8.0  2.2768842559400205e-8  \\
          9.0  2.575473721604871e-16  \\
          10.0  3.9373868288826497e-16  \\
          11.0  4.1401968961995204e-16  \\
          12.0  2.922982247672075e-16  \\
          13.0  2.9286076419884632e-16  \\
          14.0  4.3332092041038047e-16  \\
          15.0  2.922982247672075e-16  \\
          16.0  4.142879004020754e-16  \\
          17.0  2.922982247672075e-16  \\
          18.0  2.9286076419884632e-16  \\
          19.0  4.3332092041038047e-16  \\
          20.0  2.922982247672075e-16  \\
          21.0  4.142879004020754e-16  \\
          22.0  2.922982247672075e-16  \\
          23.0  2.9286076419884632e-16  \\
          24.0  4.3332092041038047e-16  \\
          25.0  2.922982247672075e-16  \\
          26.0  4.142879004020754e-16  \\
          27.0  2.922982247672075e-16  \\
          28.0  2.9286076419884632e-16  \\
          29.0  4.3332092041038047e-16  \\
          30.0  2.922982247672075e-16  \\
          31.0  4.142879004020754e-16  \\
          32.0  2.922982247672075e-16  \\
          33.0  2.9286076419884632e-16  \\
          34.0  4.3332092041038047e-16  \\
          35.0  2.922982247672075e-16  \\
          36.0  4.142879004020754e-16  \\
          37.0  2.922982247672075e-16  \\
          38.0  2.9286076419884632e-16  \\
          39.0  4.3332092041038047e-16  \\
          40.0  2.922982247672075e-16  \\
          41.0  4.142879004020754e-16  \\
          42.0  2.922982247672075e-16  \\
          43.0  2.9286076419884632e-16  \\
          44.0  4.3332092041038047e-16  \\
          45.0  2.922982247672075e-16  \\
          46.0  4.142879004020754e-16  \\
          47.0  2.922982247672075e-16  \\
          48.0  2.9286076419884632e-16  \\
          49.0  4.3332092041038047e-16  \\
          50.0  2.922982247672075e-16  \\
        }
        ;
        \addlegendentry {$\mathcal{P}_{3b}$}
        \addplot[color={rgb,1:red,0.6755;green,0.5557;blue,0.0942}, name path={5}, draw opacity={1.0}, line width={1}, solid]
        table[row sep={\\}]
        {
          \\
          1.0  16.06866296030941  \\
          2.0  3.4262600346140943  \\
          3.0  0.015284263354648244  \\
          4.0  2.1417432406797383e-6  \\
          5.0  1.2538402326696304e-14  \\
          6.0  1.1699071403104304e-15  \\
          7.0  1.0354523486918135e-15  \\
          8.0  1.0354523486918135e-15  \\
          9.0  1.0354523486918135e-15  \\
          10.0  1.0354523486918135e-15  \\
          11.0  1.0354523486918135e-15  \\
          12.0  1.0354523486918135e-15  \\
          13.0  1.0354523486918135e-15  \\
          14.0  1.0354523486918135e-15  \\
          15.0  1.0354523486918135e-15  \\
          16.0  1.0354523486918135e-15  \\
          17.0  1.0354523486918135e-15  \\
          18.0  1.0354523486918135e-15  \\
          19.0  1.0354523486918135e-15  \\
          20.0  1.0354523486918135e-15  \\
          21.0  1.0354523486918135e-15  \\
          22.0  1.0354523486918135e-15  \\
          23.0  1.0354523486918135e-15  \\
          24.0  1.0354523486918135e-15  \\
          25.0  1.0354523486918135e-15  \\
          26.0  1.0354523486918135e-15  \\
          27.0  1.0354523486918135e-15  \\
          28.0  1.0354523486918135e-15  \\
          29.0  1.0354523486918135e-15  \\
          30.0  1.0354523486918135e-15  \\
          31.0  1.0354523486918135e-15  \\
          32.0  1.0354523486918135e-15  \\
          33.0  1.0354523486918135e-15  \\
          34.0  1.0354523486918135e-15  \\
          35.0  1.0354523486918135e-15  \\
          36.0  1.0354523486918135e-15  \\
          37.0  1.0354523486918135e-15  \\
          38.0  1.0354523486918135e-15  \\
          39.0  1.0354523486918135e-15  \\
          40.0  1.0354523486918135e-15  \\
          41.0  1.0354523486918135e-15  \\
          42.0  1.0354523486918135e-15  \\
          43.0  1.0354523486918135e-15  \\
          44.0  1.0354523486918135e-15  \\
          45.0  1.0354523486918135e-15  \\
          46.0  1.0354523486918135e-15  \\
          47.0  1.0354523486918135e-15  \\
          48.0  1.0354523486918135e-15  \\
          49.0  1.0354523486918135e-15  \\
          50.0  1.0354523486918135e-15  \\
        }
        ;
        \addlegendentry {$\mathcal{P}_4$}
        \addplot[color={rgb,1:red,0.0;green,0.6658;blue,0.681}, name path={6}, draw opacity={1.0}, line width={1}, solid]
        table[row sep={\\}]
        {
          \\
          1.0  128.77261525022573  \\
          2.0  0.7911796569297432  \\
          3.0  0.0004486295601145351  \\
          4.0  9.274132759556218e-10  \\
          5.0  8.016077249960062e-16  \\
          6.0  1.0103001729651043e-15  \\
          7.0  6.897428396198435e-16  \\
          8.0  5.863022755030523e-16  \\
          9.0  4.847255317199981e-16  \\
          10.0  4.847255317199981e-16  \\
          11.0  4.847255317199981e-16  \\
          12.0  4.847255317199981e-16  \\
          13.0  4.847255317199981e-16  \\
          14.0  4.847255317199981e-16  \\
          15.0  4.847255317199981e-16  \\
          16.0  4.847255317199981e-16  \\
          17.0  4.847255317199981e-16  \\
          18.0  4.847255317199981e-16  \\
          19.0  4.847255317199981e-16  \\
          20.0  4.847255317199981e-16  \\
          21.0  4.847255317199981e-16  \\
          22.0  4.847255317199981e-16  \\
          23.0  4.847255317199981e-16  \\
          24.0  4.847255317199981e-16  \\
          25.0  4.847255317199981e-16  \\
          26.0  4.847255317199981e-16  \\
          27.0  4.847255317199981e-16  \\
          28.0  4.847255317199981e-16  \\
          29.0  4.847255317199981e-16  \\
          30.0  4.847255317199981e-16  \\
          31.0  4.847255317199981e-16  \\
          32.0  4.847255317199981e-16  \\
          33.0  4.847255317199981e-16  \\
          34.0  4.847255317199981e-16  \\
          35.0  4.847255317199981e-16  \\
          36.0  4.847255317199981e-16  \\
          37.0  4.847255317199981e-16  \\
          38.0  4.847255317199981e-16  \\
          39.0  4.847255317199981e-16  \\
          40.0  4.847255317199981e-16  \\
          41.0  4.847255317199981e-16  \\
          42.0  4.847255317199981e-16  \\
          43.0  4.847255317199981e-16  \\
          44.0  4.847255317199981e-16  \\
          45.0  4.847255317199981e-16  \\
          46.0  4.847255317199981e-16  \\
          47.0  4.847255317199981e-16  \\
          48.0  4.847255317199981e-16  \\
          49.0  4.847255317199981e-16  \\
          50.0  4.847255317199981e-16  \\
        }
        ;
        \addlegendentry {$\mathcal{P}_{4b}$}
      \end{axis}
    \end{tikzpicture}
  \end{adjustbox}
  \caption{Objective value, $\|G(x^k)\|$ over $k$ for out Algorithm applied
    to the problems from~\cite{Kub11Inte}, observing quadratic convergence for
  all instances, except for the infeasible problem}
\end{figure}
\begin{remark}
  It is interesting to note that Problem~\eqref{eq:inter_p2}
  is inconsistent and thus
  it violates the assumptions of our super-linear convergence result.
  Nonetheless we see a behavior that resembles linear convergence.
\end{remark}

Finally, as mentioned in the introduction, complementarity problems serve as
a very important source of non-smooth systems. For a complete study on
how such problems arise from constraints in mechanical systems, the
reader is invited to consult~\cite{Bro16Nons}.
In this work, we consider
a toy complementarity problem in order to showcase how our
method might behave for such problems. Our example can be formulated as
$G(x) = 0$, where
$G:\bbRn \times \bbRn \times \bbRn \to \bbRn \times \bbRn$ with
\begin{equation*}
  G(x, y, z) =
  \begin{bmatrix}
    Ax + b - y + z \\
    \min(1 - x, y)
  \end{bmatrix},
\end{equation*}
where $A \in \bbRnxn$ and $b \in \bbRn$ are random and the minimum is understood element wise.

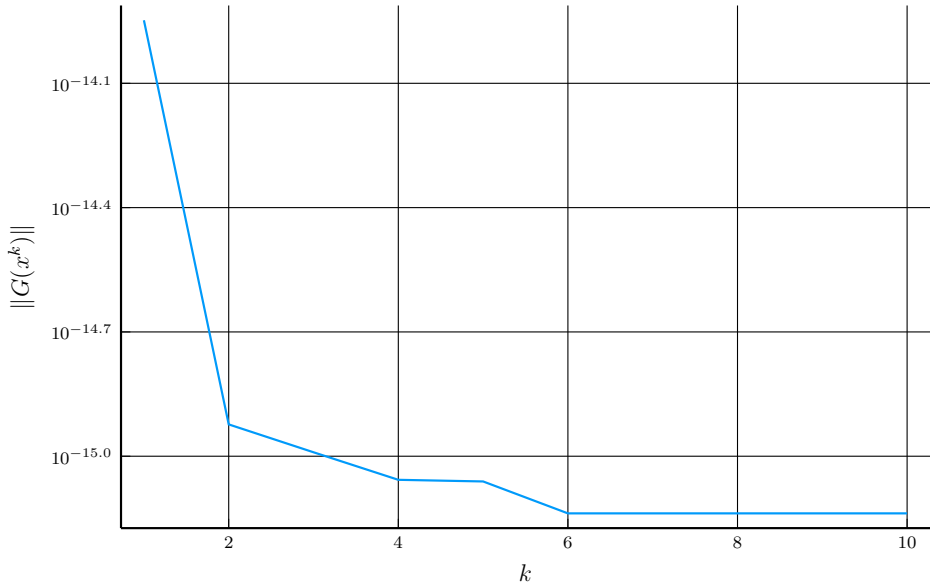
\begin{figure}[ht!]
  \begin{adjustbox}{width=\textwidth}%
    \begin{tikzpicture}[/tikz/background rectangle/.style={fill={rgb,1:red,1.0;green,1.0;blue,1.0}, fill opacity={1.0}, draw opacity={1.0}}, show background rectangle]
      \begin{axis}[point meta max={nan}, point meta min={nan}, legend cell align={left}, legend columns={1}, title={}, title style={at={{(0.5,1)}}, anchor={south}, font={{\fontsize{14 pt}{18.2 pt}\selectfont}}, color={rgb,1:red,0.0;green,0.0;blue,0.0}, draw opacity={1.0}, rotate={0.0}, align={center}}, legend style={color={rgb,1:red,0.0;green,0.0;blue,0.0}, draw opacity={1.0}, line width={1}, solid, fill={rgb,1:red,1.0;green,1.0;blue,1.0}, fill opacity={1.0}, text opacity={1.0}, font={{\fontsize{8 pt}{10.4 pt}\selectfont}}, text={rgb,1:red,0.0;green,0.0;blue,0.0}, cells={anchor={center}}, at={(1.02, 1)}, anchor={north west}}, axis background/.style={fill={rgb,1:red,1.0;green,1.0;blue,1.0}, opacity={1.0}}, anchor={north west}, xshift={1.0mm}, yshift={-1.0mm}, width={145.4mm}, height={99.6mm}, scaled x ticks={false}, xlabel={$k$}, x tick style={color={rgb,1:red,0.0;green,0.0;blue,0.0}, opacity={1.0}}, x tick label style={color={rgb,1:red,0.0;green,0.0;blue,0.0}, opacity={1.0}, rotate={0}}, xlabel style={at={(ticklabel cs:0.5)}, anchor=near ticklabel, at={{(ticklabel cs:0.5)}}, anchor={near ticklabel}, font={{\fontsize{11 pt}{14.3 pt}\selectfont}}, color={rgb,1:red,0.0;green,0.0;blue,0.0}, draw opacity={1.0}, rotate={0.0}}, xmajorgrids={true}, xmin={0.7299999999999995}, xmax={10.27}, xticklabels={{$2$,$4$,$6$,$8$,$10$}}, xtick={{2.0,4.0,6.0,8.0,10.0}}, xtick align={inside}, xticklabel style={font={{\fontsize{8 pt}{10.4 pt}\selectfont}}, color={rgb,1:red,0.0;green,0.0;blue,0.0}, draw opacity={1.0}, rotate={0.0}}, x grid style={color={rgb,1:red,0.0;green,0.0;blue,0.0}, draw opacity={0.1}, line width={0.5}, solid}, axis x line*={left}, x axis line style={color={rgb,1:red,0.0;green,0.0;blue,0.0}, draw opacity={1.0}, line width={1}, solid}, scaled y ticks={false}, ylabel={$\|G(x^k)\|$}, y tick style={color={rgb,1:red,0.0;green,0.0;blue,0.0}, opacity={1.0}}, y tick label style={color={rgb,1:red,0.0;green,0.0;blue,0.0}, opacity={1.0}, rotate={0}}, ylabel style={at={(ticklabel cs:0.5)}, anchor=near ticklabel, at={{(ticklabel cs:0.5)}}, anchor={near ticklabel}, font={{\fontsize{11 pt}{14.3 pt}\selectfont}}, color={rgb,1:red,0.0;green,0.0;blue,0.0}, draw opacity={1.0}, rotate={0.0}}, ymode={log}, log basis y={10}, ymajorgrids={true}, ymin={6.705885818553702e-16}, ymax={1.223940484848975e-14}, yticklabels={{$10^{-15.0}$,$10^{-14.7}$,$10^{-14.4}$,$10^{-14.1}$}}, ytick={{9.99999999999996e-16,1.9952623149688664e-15,3.981071705534953e-15,7.943282347242789e-15}}, ytick align={inside}, yticklabel style={font={{\fontsize{8 pt}{10.4 pt}\selectfont}}, color={rgb,1:red,0.0;green,0.0;blue,0.0}, draw opacity={1.0}, rotate={0.0}}, y grid style={color={rgb,1:red,0.0;green,0.0;blue,0.0}, draw opacity={0.1}, line width={0.5}, solid}, axis y line*={left}, y axis line style={color={rgb,1:red,0.0;green,0.0;blue,0.0}, draw opacity={1.0}, line width={1}, solid}, colorbar={false}]
        \addplot[color={rgb,1:red,0.0;green,0.6056;blue,0.9787}, name path={1}, draw opacity={1.0}, line width={1}, solid]
        table[row sep={\\}]
        {
          \\
          1.0  1.1273610168735491e-14  \\
          2.0  1.1934428697813141e-15  \\
          3.0  1.0222205069512407e-15  \\
          4.0  8.771561172376096e-16  \\
          5.0  8.691850745534256e-16  \\
          6.0  7.2803698347351e-16  \\
          7.0  7.2803698347351e-16  \\
          8.0  7.2803698347351e-16  \\
          9.0  7.2803698347351e-16  \\
          10.0  7.2803698347351e-16  \\
        }
        ;
      \end{axis}
    \end{tikzpicture}
  \end{adjustbox}
  \caption{Objective value, $\|G(x^k)\|$ over $k$ for our Algorithm applied to a toy
    model steaming from the non-smooth formulation of a complementarity constraint
    problem, observing just super-linear convergence and not quadratic convergence}
\end{figure}


All the implementations have been done in the Julia programming language and
are available at~\cite{Pin25AnewGit}.

\section{Conclusions and Further Works}
This paper shows that Newton's method can easily be extended to under-determined
problems, while the concept of Newton differentiability provides enough
regularity for super-linear convergence. A very important class of
problems fitting for our method is that of complementarity constraint problems
arising from mechanics, thus the algorithm developed in this work can provide
a valuable tool for engineers in order to develop over-parameterized models.
It is quite clear that our work can be extended by relaxing the Newton
differentiability assumption to weak Newton differentiability in order to
obtain linear convergence. Quasi-Newton methods can also fit in the framework
of Newton differentiability, and as such it would be interesting to see
Quasi-Newton approaches to solving non-smooth under-determined systems of
equations.

\section{Acknowledgments}
An earlier version of this work was part of the author's PhD thesis, completed at
the University of G\"ottingen under the supervision of D. Russell Luke. The author
would like to thank D. Russell Luke for his guidance and advice during the
undertaking of the PhD. The author would also like to thank his
postdoc mentor, Sorin-Mihai Grad.

\printbibliography{}

\end{document}